\documentclass[10pt,reqno,a4paper]{amsart}

\usepackage[pdftex]{graphicx}
\usepackage{pgfplots}
\usepackage{tikz}  
\usepackage{tikz-3dplot}
\usepackage{xifthen}
\usepackage{verbatim}
\usepackage{makecell}
\usepackage{booktabs}

\pgfplotsset{width=6cm,compat=1.18}

\usepackage{amsmath}
\usepackage{amssymb,amsthm,hyperref,caption}
\usepackage{color}
\definecolor{meinBlau}{rgb}{0.2,0.2,0.9} 
\definecolor{blau}{rgb}{0,0,0.75} 
\definecolor{rot}{rgb}{0.74,0,0} 
\hypersetup{colorlinks,linkcolor=blau,citecolor=blue,urlcolor=meinBlau}
\usepackage{times}
\usepackage{enumitem}

\allowdisplaybreaks

\newtheorem{theorem}{Theorem}
\newtheorem{lem}[theorem]{Lemma}

\newtheorem{prop}{Proposition}

\theoremstyle{definition}

\newtheorem{remark}{Remark}
\newtheorem{example}{Example}
\newtheorem{defi}{Definition}

\def\P{{\mathbb {P}}}
\def\E{{\mathbb {E}}}

\newcommand{\fallfak}[2]{\ensuremath{#1^{\underline{#2}}}}
\newcommand{\auffak}[2]{\ensuremath{#1^{\overline{#2}}}}

\newcommand{\N}{\ensuremath{\mathbb{N}}}

\newcommand{\R}{\ensuremath{\mathbb{R}}}

\DeclareMathOperator{\Beta}{B}
\DeclareMathOperator{\GG}{GenGemma}

\DeclareMathOperator{\Dir}{Dir}
\DeclareMathOperator{\Res}{\textrm{Res}}

\newcommand{\ML}{\operatorname{ML}}

\newcommand{\BML}{\operatorname{ML}}

\def\tilt{\operatorname{tilt}}

\newcommand\Polya{P\' olya}
\newcommand{\I}{\ensuremath{\mathbb{I}}}

\newcommand\field{\mathcal{F}}
\newcommand\given{\, \vert \, }

\newcommand\mw{\mathcal{W}_N}
\newcommand{\mW}[1]{\mathcal{W}_{#1}}
\newcommand{\Wn}{\mathbf{W}_{N}}
\newcommand{\Wns}{\mathbf{W}_{N,\mathbf{s}}}
\newcommand{\Wnms}{\mathbf{W}_{N-1,\mathbf{s}}}

\newcommand{\Stir}[2]{\genfrac{ \{ }{ \} }{0pt}{}{#1}{#2}}

\DeclareMathOperator{\law}{\overset{\mathcal{L}}{=}}

\DeclareMathOperator{\modop}{mod}

\begin{document}

%

\author[M.~Kuba]{Markus Kuba}
\address{Markus Kuba\\
Department Applied Mathematics and Physics\\
University of Applied Sciences - Technikum Wien\\
H\"ochst\"adtplatz 5, 1200 Wien} 


\title[On P\'olya-Young urn models and growth processes]{On P\'olya-Young urn models and growth processes}

\keywords{P\'olya-Young urn models, exact distribution, limit laws, martingale tail sum, law of the iterated logarithm, increasing trees,
chinese restaurant process, Stirling permutations}%
\subjclass[2000]{60F05, 60C05, 60G42} %

\begin{abstract}
This work is devoted to P\'olya-Young urns, a class of periodic P\'olya urns of importance in the analysis of Young tableaux. We provide several extension of the previous results of Banderier, Marchal and Wallner~\cite[Ann.~Prob.~ (2020)]{BMW2020} on P\'olya-Young urns and also generalize the previously studied model. We determine the limit law of the generalized model, involving the the local time of noise-reinforced Bessel processes. We also uncover a martingale structure, which leads directly to almost-sure convergence of the random variable of interest. This allows us to add second order asymptotics by providing a central limit theorem for the martingale tail sum, as well as a law of the iterated logarithm. We also turn to random vectors and obtain the limit law of P\'olya-Young urns with multiple colors. Additionally, we introduce several growth processes and combinatorial objects, which are closely related to urn models. We define increasing trees with periodic immigration and we related the dynamics of the P\'olya-Young urns to label-based parameters in such tree families. Furthermore, we discuss a generalization of Stirling permutations and obtain a bijection to increasing trees with periodic immigration. Finally, we introduce a chinese restaurant process with competition and relate it to increasing trees, as well as P\'olya-Young urns. 
\end{abstract}

\maketitle

\vspace{-0.5cm}

\section{Introduction}
\subsection{Periodic P\'olya-Eggenberger urn models}
In the classical P\'olya urn model, also called P\'olya-Eggenberger urn scheme~\cite{EggPo1923,EggPo1928,Mahmoud2008}, one starts with an urn containing $w_0$ white and $b_0$ black balls at time $N=0$. At every discrete time step $N\in\N$ a single ball is drawn uniformly at random. Its color is inspected and the drawn ball is reinserted into the urn. The dynamics of the urn process proceed according to the observed color: if the ball is white, $a$ white and $b$ black balls are added into the urn; if the ball is black, $c$ black balls and $d$
white balls are added. Here, $a$, $b$, $c$, $d$ are usually non-negative integers; in some models the conditions are relaxed to non-negative real values; also certain negative values can be considered~\cite{Davidson2018}. It is usually assumed that the urn is tenable, the process of sampling balls and adding/replacing balls can be continued ad infinitum. This process is encoded the so-called ball replacement matrix, sometimes also called transition matrix:
\[
M = 
\left(
\begin{matrix}
a& b\\
c& d\\
\end{matrix}
\right)
\]
An urn model and its associated replacement matrix $M$ is called balanced, if the total number of added balls is constant,
$\sigma=a+b=c+d$. Thus, the total number of balls $T_N$ after $N$ steps is deterministic,
\[
T_N=w_0+b_0 + \sigma \cdot N.
\]
Quantities of of interest are the number of white balls $W_N$ after~$N$ draws, and the number of black balls~$B_N$ after~$N$ draws, 
related by $W_N+B_N=T_N$. P\'olya urn schemes are useful mathematical models for growth processes with
a wide range of applications in several areas. This includes the analysis of random trees,
graphs and algorithms, population genetics and the spread of epidemics. For a discussion of these applications and further information, we refer to the monographs of Johnson and Kotz~\cite{JoKo1977}, Mahmoud~\cite{Mahmoud2008}, as well as the article of Janson~\cite{Janson2004}.

\smallskip

Recently, the following extension of the P\'olya-Eggenberger urn model has been given. 
\begin{defi}[Periodic P\'olya urns~\cite{BMW2020}]
A periodic P\'olya urn of period $p$ with replacement matrices
$M_1, M_2,\dots,M_p$ is a variant of a P\'olya urn in which the replacement matrix
$M_k$ is used at steps $N=np+k$. Such a model is called balanced if each of its replacement
matrices is balanced.
\end{defi}
For $p=1$ such models reduce to the classical P\'olya urn model. A certain class of periodic P\'olya urns called P\'olya-Young urns is of great interest due to its connection to Young tableaux. More precisely, 
with the help of what the so-called density method, results for 
P\'olya-Young urns allow to study the distribution of the southeast
corner of a triangular Young tableau~\cite{BMW2018,BMW2020,BW2021}. This has important implications for the analysis of the continuous limit of large random Young tableaux and has links with random surfaces, see the impressive work~\cite{BMW2020}. The analysis of P\'olya-Young urns was carried out in~\cite{BMW2020} using analytic combinatorics and the method of moments; in more detail, they used a generating function approach based on the history counting technique, leading to solvable differential equations (see~\cite{Dumas2006} for details about this method), which where used to obtain asymptotics of the moments. 

\smallskip 

We consider here a generalization of this class, with an additional positive real parameter $\sigma$ and relaxed conditions on other parameters.

\begin{defi}[Generalized P\'olya-Young urns]
Let $p\in\N$ and $\sigma,\ell>0$ denote positive reals. A Young–P\'olya urn of period $p$, balance $\sigma$ 
and parameter $\ell$ is a periodic P\'olya urn of period $p$ (with $w_0>0$ to avoid degenerate
cases) and replacement matrices
\[
M_1=M_2=\dots=M_{p-1}=
\left(
\begin{matrix}
\sigma&0\\
0&\sigma\\
\end{matrix}
\right),
\quad\text{and }
M_{p}=
\left(
\begin{matrix}
\sigma&\ell\\
0&\sigma+\ell\\
\end{matrix}
\right).
\]
\end{defi}
We note that in the previous work $\sigma$ was restricted to the value $\sigma=1$, and also only $\ell\in\N$ was considered.
The random variable of interest is the number of white balls $W_N$ after $N$ draws, as well as its limit law for $N$ tending to infinity.

\smallskip

In the following we provide an analysis of generalized P\'olya-Young urns, extending the previous results in several ways. 
We uncover a martingale structure, 
leading directly to almost sure convergence of the scaled random variable.
We also obtain simple explicit formulas for the moments and also the probability generating function. We obtain a limit law for the number of white balls and also moment convergence. In the special case treated earlier~\cite{BMW2020}, the limit law was described in terms of the Beta distribution and generalized Gamma random variables. For the generalized Young–P\'olya urns, we observe a different nature of the limit law. It is related to the local time of the noise-reinforced Bessel process~\cite{Bertoin2022}. For special choices of $\sigma$, $p$ and $\ell$ we also obtain product distributions, involving the Beta distribution and generalized Gamma random variables. 

\smallskip

We also provide new second order asymptotics, analyzing martingale tail sums. We remind the reader that for a martingale $(\mw)$ converging almost surely to a random variable $\mW{\infty}$, the sequence $(\mw - \mW{\infty})$ is called martingale tail sum. 
Informally, our results for generalized P\'olya-Young urns reads as
\[
W_N \sim \kappa N^{\Lambda}\cdot \mW{\infty} + \beta^{-1}\sqrt{\mW{\infty}}\cdot N^{\frac{\Lambda}2}\cdot \mathcal{N},
\]
where $\mW{\infty}$ is the almost sure limit of $g_N W_N$, where $\Lambda$ and $g_N$ are defined in Theorem~\ref{the1}, or in other words, a multiple of the first order limit of $W_N / N^{\Lambda}$. Here, we write $X_N \sim Y_N+Z_N\mathcal{N}$ when $Z_N^{-1}(X_N-Y_N)\to \mathcal{N}$ in distribution, and $\mathcal{N}$ denotes a zero-mean normal random variable with variance one.
Moreover, we also obtain a law of the iterated logarithm. We rely mostly on the theory of discrete martingales, as well as tools from discrete mathematics. Furthermore, we extend P\'olya-Young models to urns with multiple colors and provide an analysis of the corresponding random vector, counting the number of balls of different colors. Concerning applications, generalized P\'olya-Young urns are then related to a new model of increasing trees, called increasing trees with immigration, and parameters appearing in this new tree models. Furthermore, we discuss also a generalization of Stirling permutations and obtain a bijection to increasing trees with periodic immigration. We define a chinese restaurant process with competition and relate it to increasing trees, as well as P\'olya-Young urns. Finally, we also consider a chinese restaurant process with competition and a cocktail bar.

\subsection{Notation}
Throughout this work we denote with $\N$ the positive integers and with $\N_0$ the non-negative integers. 
We use the notation $\auffak{x}{s}=x(x+1)\dots(x+s-1)$ for the rising factorials, $s\in\N_0$, with $\auffak{x}0=1$.

\section{Preliminaries}
\label{subsec:prelim} 
We collect here the basic properties of several distributions, 
needed later for the identification of limit laws appearing subsequently. We encounter classical distribution, such as the Beta distribution
and the (generalized) Gamma distribution, but also new laws like the local time of a noise reinforced Bessel process, introduced by Bertoin~\cite{Bertoin2022}, and the three parameter Mittag-Leffler law~\cite{BKW2022,Moehle2021}.

\begin{defi}
The beta-distribution $\beta(a,b)$ with real parameters $a,b>0$ is defined by the density function
\[
f(x)=\frac{x^{a-1}(1-x)^{b-1}}{\Beta(a,b)},\quad x\in(0,1),
\]
with $\Beta(a,b)$ denoting the Beta-function.
\end{defi}
The sequence $(\mu_s)$ of raw moments of a beta-distribution satisfies
\[
\E(X^s)=\frac{\Gamma(a+s)\Gamma(a+b)}{\Gamma(a)\Gamma(a+b+s)}.
\]
\begin{defi}
The generalized gamma distribution $\GG(a, b)$ with
with real parameters $a,b>0$ is defined by the density function
\[
f(x)=\frac{b x^{a-1}e^{-x^{b}}}{\Gamma(a/b)},\quad x\in(0,\infty)
\]
where $\Gamma$ is the classical gamma function.
\end{defi}
We note that the generalized gamma distribution can be obtained from the classical Gamma distribution
taking a suitable power. We also point out that the Generalized Gamma distribution with parameters $a,b$ is determined by its sequence of raw moments $(\mu_s)$:
\[
\mu_s=\frac{\Gamma(\frac{a+s}{b})}{\Gamma(\frac{a}{b})},\quad s\ge 0.
\]

\smallskip

For a random variable $X$ with density function $f(x)$, the tilt of $f(x)$ by a nonnegative integrable function $g(x)$ is the density 
\[
\frac{g(x)}{\E(g(X))} \cdot f(x).
\]
An important class of tilted densities are the \textit{polynomially tilted densities}, where one tilts by a polynomial $g(x)=x^c$(with $c$ being any real value such that $\E(X^c)$ is well defined).
We then use the notation
\begin{equation*}
\tilt_c(f(x))=\frac{x^c}{\E(X^c)} \cdot f(x). 
\end{equation*}

Such tilted densities occur in many places in the literature, see~\cite{BKW2022} and the references therein. They are also called~\cite{arratia2018size,Bertoin2022} size-biased distributions (case $c=1$); see~\cite{arratia2018size} for a general discussion of biased distributions. We note in passing that many classes of distributions like the beta distribution, generalized gamma distribution, the $F$-distribution, the beta-prime distribution, and distributions with gamma-type moments~\cite{Jan2010} are closed under the tilting operation. 
The following lemma shows that the operator $\tilt_c$ admits in fact several equivalent definitions
using the density, the moments, or the Laplace transform.

\begin{lem}[Polynomially tilted density functions and moment shifts~\cite{BKW2022}]
\label{LemMain}
Consider a random variable $X$ with moment sequence $(\mu_s)_{s\ge 0}$ and density $f(x)$ with support $[0,\infty)$. 
Now consider a random variable $X_c$ with $c\in\N$, having a distribution uniquely determined by its moments. Then the following properties are equivalent:
\begin{enumerate}
	\item\label{ita} Tilted density: $X_c$ is a random variable with density $f_c(x) = \frac{x^c}{\mu_c}\cdot f(x)$.
	\item\label{itb} Shifted moments: $X_c$ is a random variable with moments $\E(X_c^s)=\frac{\mu_{s+c}}{\mu_c}$. 
	\item\label{itc} Differentiated moment generating function: $X_c$ satisfies 
	$
	\E(e^{tX_c})=\frac{1}{\mu_c}\frac{d^c}{dt^c}\E(e^{tX}).
	$
\end{enumerate}
We write $\tilt_c(X) = X_c$ for corresponding tilted random variable. 
\end{lem}

We also note the following property:
\begin{align}
\tilt_{c_1}\big(\tilt_{c_2}(X)\big) = \tilt_{c_2}\big(\tilt_{c_1}(X)\big)
=\tilt_{c_1+c_2}(X)=X_{c_1+c_2}.
\label{eq:tilt}
\end{align}
This follows by direct computation, for example
\begin{align*}
\tilt_{c_1}\big(\tilt_{c_2}(f)\big)
&=\tilt_{c_1}\Big(\frac{x^{c_2}}{\mu_{c_2}}\cdot f(x)\Big)
= \frac{x^{c_1+c_2}}{\frac{\mu_{c_1+c_2}}{\mu_{c_2}}\cdot\mu_{c_2}}\cdot f(x)\\
&=\frac{x^{c_1+c_2}}{\mu_{c_1+c_2}}\cdot f(x)
=\tilt_{c_1+c_2}\big( f(x)\big).
\end{align*}

\begin{remark}[Tilt with $c\in \R$]
\label{rem:tilt}
One may extend the equivalence of the properties~\eqref{ita} and~\eqref{itb} to $c\in \R$, assuming that the corresponding moments exist. 
For non-integer real $c$ property~\eqref{itc} involves the fractional derivative.
\end{remark}

Next, we apply the tilt-operator to a random variable arising in the context of a noise reinforced Bessel process
of dimension $0< d <2$ and with reinforcement parameter $r\in(-\infty,1/2)$\footnote{Bertoin uses the letter $p$ for the reinforcement, which is reserved in our work for the periodicity parameter of P\'olya-Young urns}. Bertoin~\cite{Bertoin2022} studied the local time $\hat{L}_t$ of the noise reinforced Bessel process with parameters $d$ and $r$ at level $t$ and obtained various properties of it. This random variable also appears as a limit law in one-sided tree destruction procedures~\cite{FKP2006,KuPa2023}, as well as in the context of so-called $\alpha$-sun distributions~\cite{Simon2023}. In the following we collect a result on the structure of the moments of $\hat{L}_t$ based on~\cite{Bertoin2022}.
\begin{lem}[Moments of the local time of a noise reinforced Bessel process]
\label{lem:PanKu1}
Let $T$ denote a random variable with raw moments
\[
\E(T^s)=\prod_{j=1}^{s}j\cdot\frac{\Gamma(j\beta)}{\Gamma(\alpha+j\beta)}=\Gamma(s+1)\cdot\prod_{j=1}^{s}\frac{\Gamma(j\beta)}{\Gamma(\alpha+j\beta)}.
\]
Let $L=\hat{L}_t/\kappa$ denote the scaled local time of the noise reinforced Bessel process with parameters $d$ and $r$ at level $t$, 
with scale parameter
\[
\kappa=\frac{(2t)^{\alpha}\cdot\Gamma(1+\alpha)}{(1-2r)^{\alpha}\cdot\Gamma(1-\alpha)}.
\]
Then, $T\law \tilt_1(L)$. The parameters $\alpha,\beta$ of $T$ are related to the dimension $d$
and the reinforcement parameter $r$ of the process by
\begin{equation}
\label{eqn:alphaBeta}
\alpha=1-d/2\in(0,1) \text{ and } \beta=\frac{\alpha}{1-2r}>0, 
\end{equation}
or equivalently
\[
d=2(1-\alpha),\quad r=\frac12 - \frac{\alpha}{2\beta}.
\]
\end{lem}

We refer the reader to Bertoin~\cite{Bertoin2022} and Simon~\cite{Simon2023} for more properties of $\hat{L}_t$, as well as $T$, 
and a different characterizations of $T$ as an exponential functional~\cite[Corollary 4.3]{Bertoin2022}. We note that for $\alpha=\beta$, $r=0$, the random variable $L=\hat{L}_t/\kappa$ 
reduces to the classical Mittag-Leffler distribution $\ML(\alpha)=
S_\alpha^{-\alpha}$, given in terms of a positive stable law of parameter $\alpha\in(0, 1)$, see~\cite{DarlingKac1957,Feller1949,Feller}, with moment sequence $\mu_s=\Gamma(s+1)/\Gamma(s\alpha+1)$. The special case $\alpha=\frac12$ leads to a half-normal distribution for $L$, 
and thus $T\law \tilt_1(L)$ is given by a Rayleigh law. For the reader's convenience we include the short proof of Lemma~\ref{lem:PanKu1}.
\begin{proof}
The identification is directly done using moment sequences appearing in~\cite{Bertoin2022} (see also \cite{KuPa2023}). There, a noise reinforced Bessel process of dimension $0<d <2$ and with reinforcement parameter $r\in(-\infty,1/2)$ was considered. Using the parametrization~\eqref{eqn:alphaBeta}, the local time $\hat{L}_t$ of the noise reinforced Bessel process has power moments (see \cite[Theorem 1.2]{Bertoin2022}), which can be written as
\[
\E(\hat{L_t}^s)=\kappa(r,t)^s\cdot \frac{1-2r}{\Gamma(1+\alpha)}\cdot \Gamma(s)\cdot\prod_{j=1}^{s-1}\frac{\Gamma(j\beta)}{\Gamma(\alpha+j\beta)},\quad s\ge 1,
\]
where the scale factor $\kappa=\kappa(\alpha,r,t)$ is given by
\begin{equation}
\label{scale}
\kappa=\frac{(2t)^{\alpha}\cdot\Gamma(1+\alpha)}{(1-2r)^{\alpha}\cdot\Gamma(1-\alpha)}.
\end{equation}
This implies that the scaled random variable $L=\hat{L_t}/\kappa$ has moment sequence 
$\mu_s=\E(L^s)$, given by
\[
\mu_s=\frac{1-2r}{\Gamma(1+\alpha)}\cdot \Gamma(s)\cdot\prod_{j=1}^{s-1}\frac{\Gamma(j\beta)}{\Gamma(\alpha+j\beta)}, \quad s\ge 1.
\]
We readily observe that the random variable $T=\tilt_1(L)$ has the stated moment sequence $\E(T^s)=\frac{\mu_{s+1}}{\mu_1}$.
\end{proof}

Next, we add another simple but very useful observation.
\begin{lem}
\label{lem:PanKu2}
Given the random variable $T\law \tilt_1(L)$, where $L$ denotes the scaled local time of a noise-reinforced Bessel process (see Lemma~\ref{lem:PanKu1}, satisfying $\beta=\frac{\alpha}{m}$ for $m\in\N$. 
Then, the sequence of raw moments of $T$ can be written as a product, ranging from one to $m$:
\[
\E(T^s)=\Gamma(s+1)\cdot\prod_{j=1}^{m}\frac{\Gamma(\frac{j\alpha}{m})}{\Gamma(\frac{(s+j)\alpha}{m})},\quad s\ge 1.
\]
\end{lem}
\begin{proof}
As the observation was stated in~\cite{KuPa2023} without a proof, we add the simple computations. For $s\le m$ we observe that
\begin{align*}
\prod_{j=1}^{s}\frac{\Gamma(j\beta)}{\Gamma(\alpha+j\beta)}&=\prod_{j=1}^{s}\frac{\Gamma(\frac{\alpha j}m)}{\Gamma(\frac{\alpha(m+j)}{m})}=\bigg(\prod_{j=1}^{s}\frac{\Gamma(\frac{\alpha j}m)}{\Gamma(\frac{\alpha(m+j)}{m})}\bigg)
\cdot\prod_{j=s+1}^{m}\frac{\Gamma(\frac{\alpha j}m)}{\Gamma(\frac{\alpha j}m)}\\
&=\frac{\prod_{j=1}^{m}\Gamma(\frac{\alpha j}m)}{\big(\prod_{j=1}^{s}\Gamma(\frac{\alpha(m+j)}{m})\prod_{k=s+1}^{m}\Gamma(\frac{\alpha k}m)}\Big).
\end{align*}
Regrouping the denominator gives the stated result. For $s>m$ we have:
\begin{align*}
\prod_{j=1}^{s}\frac{\Gamma(\frac{\alpha j}m)}{\Gamma(\frac{\alpha(m+j)}{m})}
=\frac{\prod_{j=1}^{m}\Gamma(\frac{\alpha j}m)}{\prod_{j=s+1-m}^{s}\Gamma(\frac{\alpha(m+j)}{m})}
\cdot \frac{\prod_{j=m+1}^{s}\Gamma(\frac{\alpha j}m)}{\prod_{j=1}^{s-m}\Gamma(\frac{\alpha(m+j)}{m})}.
\end{align*}
Shifting the index and cancellation of the latter products gives the desired representation.
\end{proof}

Finally, we collect a recently discovered distribution~\cite{BKW2022,Moehle2021}, which also appeared earlier in the literature~\cite{Dumas2006}. It generalizes the classical Mittag-Leffler distribution, as well as the two-parameter Mittag-Leffler law~\cite{GoldschmidtHaas2015,James2015}.

\begin{defi}[Three-parameter Mittag-Leffler distribution~\cite{BKW2022,Moehle2021}]\label{def:BetaMittagLeffler}
The \textit{three-parameter Mittag-Leffler distribution} $\BML(\alpha,\beta,\gamma)$ is defined as the distribution of the product of independent random variables
	\begin{align*}
	Z \law Y \cdot B^{\alpha}
	\end{align*}
where $Y \law \tilt_{\beta/\alpha}(S_\alpha^{-\alpha})$ denotes a moment-tilted positive stable law of parameter $\alpha\in(0, 1)$
and $B \law \operatorname{Beta}(\beta,\gamma)$, with $\beta>0$, and $\gamma\geq 0$. 
\end{defi}
We note that the three-parameter Mittag-Leffler distribution $\BML(\alpha,\beta,\gamma)$, 
also has the following moments of order $s$ (for $s>0$):
	\begin{equation}
	\label{eq:momBML}
	\E(Z^s)=
	\frac{\Gamma\left(s+\frac{\beta}{\alpha}\right)\Gamma\left(\beta+\gamma\right)}{\Gamma\left(\alpha s+\beta+\gamma\right)\Gamma\left(\frac{\beta}{\alpha}\right)}.
\end{equation}

\section{Analysis of generalized P\'olya-Young urns} 
Throughout this work we use discrete martingales, a classical method used in the analysis of urn models~\cite{Gouet1989,Gouet93,Gouet1997,JKP2011,KuPa2014,Mor2005}. We refer the reader to the book of Williams~\cite{williams1991probability} for a gentle introduction to this topic and to the book of Hall and Heyde~\cite{hallheyde80} for more advanced content.
\subsection{Martingale structure and moments}
First, we emphasize that the total number of balls $T_N$ in a P\'olya-Young urn at time $N=np+k$, with $n\in\N_0$, $0\le k\le p-1$ is given by
\begin{equation}
\label{eq:TotalPoYou}
T_N=w_0+b_0+\big(n(p-1)+k\big)\sigma +n(\sigma+\ell)
=n(p\sigma + \ell)+k\sigma +w_0+b_0.
\end{equation}

\begin{theorem}[Distribution, martingale and almost sure limit]\label{the1}
The $s$th rising factorial moments of the random variable $W_N$,
counting the number of white balls in a generalized P\'olya-Young urn, are given by 
\[
\E(\auffak{(W_N/\sigma)}{s})=\auffak{(w_0/\sigma)}{s}\prod_{j=0}^{N-1}\frac{T_{j}+s\sigma}{T_{j}}.
\]
The probability mass function of $W_N/\sigma$ is determined by the probability generating function, given by 
\[
\E(v^{W_N/\sigma})=\sum_{s\ge 0}\frac{(v-1)^s}{s}
\sum_{r=0}^{s}L_{s,r}(-1)^{s-r}\auffak{(w_0/\sigma)}{r}\prod_{j=0}^{N-1}\frac{T_{j}+r\sigma}{T_{j}},
\]
where $L_{s,r}=\binom{s}{r}\frac{(s-1)!}{(r-1)!}$ denote the Lah numbers. Let $N=np+k$, $0\le k\le p$ and let $n\to\infty$.
\begin{equation}
\label{eq:Def_gN}
g_N= \prod_{j=0}^{N-1}\frac{T_{j}}{T_{j}+\sigma},\quad
g_N\sim \kappa \cdot N^{-\Lambda} ,\quad \kappa= p^{\Lambda}\prod_{r=0}^{p-1}\frac{\Gamma(\frac{r}\psi  +\frac{w_0+b_0}{\sigma\psi}+\frac{1}\psi)}{\Gamma(\frac{r}\psi +\frac{w_0+b_0}{\sigma \psi})},
\end{equation}
with $\psi=p+\ell/\sigma$ and $\Lambda=\frac{p}{\psi}=\frac{p}{p+\ell/\sigma}$. 

\smallskip

The random variable $\mw=g_N\cdot W_N$ is a non-negative martingale 
and converges almost surely to a limit $\mW{\infty}$. 
Furthermore, the random variable $W_N/N^{\Lambda}$ satisfies 
\[
W_N/N^{\Lambda}\to \frac{1}{\kappa}\mW{\infty}. 
\]
\end{theorem}

\begin{remark}[P\'olya-Young urns and P\'olya urns]
For $\ell\to 0$ the P\'olya-Young urn models degenerate to ordinary P\'olya urns. Similarly, for $p=1$ we also reobtain the P\'olya urns. One readily may perform a sanity check: 
for $\ell\to 0$ we have $\psi\to p$, as well as $\Lambda\to 1$. Additionally, in Theorem~\ref{the2} we observe that for $\ell\to 0$
we re-obtain the classical Beta limit law of ordinary P\'olya urns.
\end{remark}

\begin{remark}[Universality of explicit expressions: triangular models]
\label{rem:universal1}
Interestingly, the explicit expressions persist for a much more general class of urn models.
Following~\cite{BMW2020} we introduce generalized Young–P\'olya urns of period $p$, balance $\sigma$ and
parameters $\ell_j\ge 0$, with replacement matrices
\[
M_j=
\left(
\begin{matrix}
\sigma&\ell_j\\
0&\sigma+\ell_j\\
\end{matrix}
\right),\quad 1\le j\le p.
\]
For all such urns Theorem~\ref{the1} stays true, except for the formula for $T_N$ and the resulting asymptotic expansion of $g_N$. 
This can be directly seen from the proof of Theorem~\ref{the1}, as the parameters $\ell_j$ only change the value of $T_N$.
Similarly, the methods applied subsequently for the proof of Theorem~\ref{the2} can be used to obtain limit laws for this general class.
\end{remark}

\begin{remark}[Universality of the martingale structure]
\label{rem:universal2}
We emphasize that the martingale structure persists for even much more general class.
Following~\cite{BMW2020} we introduce generalized Young–P\'olya urns of period $p$, balances $\sigma_j>0$ and
parameters $\ell_j\ge 0$, with replacement matrices
\[
M_j=
\left(
\begin{matrix}
\sigma_j&\ell_j\\
0&\sigma_j+\ell_j\\
\end{matrix}
\right),\quad 1\le j\le p.
\]
For all such urns we can define a value $g_N$ such that
$\mw=g_N\cdot W_N$ is a non-negative martingale. However, the moments turn out to be more complicated compared to Theorem~\ref{the1}.
\end{remark}

The limit law $\mW{\infty}$ and the closely related limit law of $W_N\cdot N^{-\Lambda}$ 
are readily obtained using the explicit expressions of the rising factorial moments.
Similar to~\cite{BMW2020}, we the limit law is related to the Beta distribution and generalized Gamma random variables.

\begin{theorem}[Moments of the limit law, density and characterization]\label{the2}
The raw moments of $W_N/N^{\Lambda}$ converge for $N\to\infty$. For $s\ge 1$ we have
\[
\E(W_N^s/\sigma^s N^{s\Lambda})\to \mu_s=\frac{\Gamma(s+\frac{w_0}\sigma)}{\Gamma(\frac{w_0}\sigma)}\prod_{r=0}^{p-1}\frac{\Gamma(\frac{r}\psi +\frac{w_0+b_0}{\psi\sigma})}{\Gamma(\frac{r}\psi +\frac{w_0+b_0}{\psi\sigma}+\frac{s}{\psi})}.
\]
The limit law $\frac{1}{\sigma\kappa}\mW{\infty}$ is uniquely determined by its moment sequence $(\mu_s)$. It has a density $f(x)$ given by
\[
f(x)=
\frac{\prod_{r=0}^{p-1}\Gamma(\frac{r}\psi +\frac{w_0+b_0}{\psi\sigma})}{\Gamma(\frac{w_0}\sigma)}
\cdot\sum_{j\ge 0}\frac{(-1)^j}{j!\prod_{\ell=0}^{p-1}\Gamma(\frac{\ell}\psi +\frac{b_0}{\psi\sigma}-\frac{j}{\psi})}\ x^{j+w_0/\sigma-1},
\quad x>0.
\]
For $\ell/\sigma\in\N$, the limit law can be decomposed into a product of a Beta law and generalized Gamma distributions, all being mutually independent:
\[
\frac{1}{\sigma\kappa}\mW{\infty}= \psi^{\psi} B(\frac{w_0}\sigma,\frac{b_0}\sigma)\cdot
\prod_{\ell=p}^{\psi-1}GG(\frac{w_0+b_0+\ell\sigma}{\sigma},\psi).
\]
In the general case the random variable can be presented in terms of the scaled local time $L$ of a noise-reinforced Bessel process, 
of dimension 
$d=\frac{2\ell}{p\sigma +\ell}$ and reinforcement parameter $r=-\frac{p-1}2$, and a Beta distribution:
\[
\frac{1}{\sigma\kappa}\mW{\infty}=  B(\frac{w_0}\sigma,\frac{b_0}{\sigma})\cdot \tilt_{\frac{w_0+b_0}{\sigma}}(L).
\]
\end{theorem}

\begin{proof}[Proof of Theorem~\ref{the1}]
Our analysis follows~\cite{JKP2011,KuPa2014,Mor2005}. We note first that for arbitrary $N=np+k\ge 1$
\[
\E\bigl(W_N \given \field_{N-1}\bigr)= \E\bigl(W_{N-1}+\Delta_N \given \field_{N-1}\bigr)
=W_{N-1}+\sigma\frac{W_{N-1}}{T_{N-1}}=\frac{T_{N-1}+\sigma}{T_{N-1}}\cdot W_{N-1}.
\]
Interestingly, the nature of the particular urn - triangular or diagonal - is solely hidden in $T_{N-1}$.
Thus, the random variable $\mw:= g_N \cdot W_N$ satisfies
\begin{equation}
\label{eq:Martin1}
\E\bigl(\mw\given \field_{N-1}\bigr)=g_N \E\bigl(W_N \given \field_{N-1}\bigr)
=g_N\frac{T_{N-1}+\sigma}{T_{N-1}}\cdot W_{N-1}=g_{N-1}W_{N-1}=\mW{N-1},
\end{equation}
proving the martingale structure. In order to obtain explicit expressions for the moments and the probability mass function
we proceed by using again the martingale structure in another way. We consider the quantity $\binom{W_N/\sigma +s-1}{s}$. 
First, we note that
\[
\frac{W_N}\sigma=\frac{W_{N-1}+\Delta_N}\sigma=
\frac{W_{N-1}}\sigma+\delta_N,
\]
where $\delta_N$ is supported on $\{0,1\}$. 
We have
\begin{align*}
\E\bigl(\binom{\frac{W_N}\sigma +s-1}{s}\given \field_{N-1}\bigr)
&=\binom{\frac{W_{N-1}}\sigma +s-1}{s}\\
&+\E\bigl(\binom{\frac{W_{N-1}}\sigma +\delta_N +s-1}{s}-\binom{\frac{W_{N-1}}\sigma +s-1}{s}\given \field_{N-1}\bigr).
\end{align*}
By the basic identity 
\[
\binom{c}{s}=\binom{c-1}{s}+\binom{c-1}{s-1},
\]
we get further
\[
\E\bigl(\binom{\frac{W_N}\sigma +s-1}{s}\given \field_{N-1}\bigr)
=\binom{\frac{W_{N-1}}\sigma +s-1}{s}+\frac{W_{N-1}}{T_{N-1}}\binom{\frac{W_{N-1}}\sigma +s-1}{s-1}.
\]
Further simplifications give
\[
\E\bigl(\binom{\frac{W_N}\sigma +s-1}{s}\given \field_{N-1}\bigr)
=\frac{T_{N-1}+s\sigma}{T_{N-1}}\binom{\frac{W_{N-1}}\sigma +s-1}{s}.
\]
Taking expectations gives the recurrence relation
\[
\E\bigl(\binom{\frac{W_N}\sigma +s-1}{s}\bigr)
=\frac{T_{N-1}+s\sigma}{T_{N-1}}\cdot\E\bigl(\binom{\frac{W_{N-1}}\sigma +s-1}{s}\bigr),\quad N\ge 1,
\]
leading to the result
\[
\E\bigl(\binom{\frac{W_N}\sigma +s-1}{s}\bigr)
=\binom{\frac{w_0}\sigma +s-1}{s}\prod_{j=0}^{N-1}\frac{T_{j}+s\sigma}{T_{j}}.
\]
Turning to rising factorials, we get
\[
\E(\auffak{(W_N/\sigma)}{s})=\auffak{(w_0/\sigma)}{s}\prod_{j=0}^{N-1}\frac{T_{j}+s\sigma}{T_{j}}.
\]
Finally, we turn to the probability mass function of $W_N$.
The probability generating function of $W_N/\sigma$, expanded around $v=1$, is given by
\[
\E(v^{W_N/\sigma})
=\sum_{s\ge 0}\frac{\E(\fallfak{(W_N/\sigma)}{s})}{s!} (v-1)^s.
\]
We converting the falling to the rising factorials and use the binomial theorem. We get
\[
\fallfak{x}{s}=\auffak{(x-(s-1))}{s} 
=\sum_{r=0}^{s}\binom{s}{r}\auffak{x}{r}\auffak{ (-(s-1))}{s-r}
=\sum_{r=0}^{s}\binom{s}{r}\auffak{x}{r}(-1)^{s-r}\frac{(s-1)!}{(r-1)!}.
\]
The number $L_{s,r}=\binom{s}{r}\frac{(s-1)!}{(r-1)!}$ are also known as the Lah numbers, 
leading to the stated result. 

Finally, we turn to the asymptotics of $g_N$. From~\eqref{eq:TotalPoYou}
we observe that
\[
T_N/\sigma=n(p + \frac{\ell}{\sigma})+k +\frac{w_0+b_0}{\sigma},\quad N=np+k, \quad 0\le k\le p-1.
\]
This gives for $N=np+k$, where we use the shorthand notation $\psi=p + \frac{\ell}{\sigma}$.
\begin{equation}
\begin{split}
\label{eq:AsympG}
g_N&
=\bigg(\prod_{j=0}^{n-1}\prod_{r=0}^{p-1}\frac{j(p + \frac{\ell}{\sigma})+r +\frac{w_0+b_0}{\sigma}}{j(p + \frac{\ell}{\sigma})+r +\frac{w_0+b_0}{\sigma}+1}\bigg)
\cdot \prod_{r=0}^{k-1}\frac{n(p + \frac{\ell}{\sigma})+r +\frac{w_0+b_0}{\sigma}}{n(p + \frac{\ell}{\sigma})+r +\frac{w_0+b_0}{\sigma}+1}\\
&=\bigg(\prod_{j=0}^{n-1}\prod_{r=0}^{p-1}\frac{j+\frac{r}\psi +\frac{w_0+b_0}{\sigma \psi}}{j+\frac{r}\psi  +\frac{w_0+b_0}{\sigma\psi}+\frac{1}\psi}\bigg)
\cdot \prod_{r=0}^{k-1}\frac{n+\frac{r}\psi +\frac{w_0+b_0}{\sigma \psi}}{n+\frac{r}\psi  +\frac{w_0+b_0}{\sigma\psi}+\frac{1}\psi}.
\end{split}
\end{equation}
Writing the products in terms of the Gamma function gives
\begin{equation}
\begin{split}
\label{eq:AsympG2}
g_N&
=\bigg(\prod_{r=0}^{p-1}\frac{\Gamma(n+\frac{r}\psi +\frac{w_0+b_0}{\sigma \psi})\Gamma(\frac{r}\psi  +\frac{w_0+b_0}{\sigma\psi}+\frac{1}\psi)}{\Gamma(\frac{r}\psi +\frac{w_0+b_0}{\sigma \psi})\Gamma(n+\frac{r}\psi  +\frac{w_0+b_0}{\sigma\psi}+\frac{1}\psi})\bigg)
\cdot \prod_{r=0}^{k-1}\frac{n+\frac{r}\psi +\frac{w_0+b_0}{\sigma \psi}}{n+\frac{r}\psi  +\frac{w_0+b_0}{\sigma\psi}+\frac{1}\psi}.
\end{split}
\end{equation}
Using Stirling's formula for the Gamma function
\begin{equation}
\label{eq:stirling}
    \Gamma(z) = \Bigl(\frac{z}{e}\Bigr)^{z}\frac{\sqrt{2\pi }}{\sqrt{z}}%
    \Bigl(1+\frac{1}{12z}+\frac{1}{288z^{2}}+\mathcal{O}(\frac{1}{z^{3}})\Bigr),
\end{equation}
we obtain the asymptotics stated in Theorem 1, noting that
$n^{-\Lambda}\sim p^{\Lambda}N^{-\Lambda}$.
Finally, we note that by the asymptotics of $g_N$ from Theorem~\ref{the1}
the almost sure convergence of $W_N/N^{\Lambda}$ follows.
\end{proof}

Next we turn to the moments of the limit law and its identification.

\begin{proof}[Proof of Theorem~\ref{the2}]
Our starting point for the moments is the expression for the rising factorial moments of Theorem~\ref{the1}. We convert the raw moments to the rising factorial moments using the Stirling numbers of the second kind $\Stir{s}{r}$, 
\begin{equation}
\label{eq:risingToRaw}
x^s = \sum_{r=0}^{s}\Stir{s}{r}(-1)^{s-r}\auffak{x}{r}.
\end{equation}
This implies that
\begin{equation}
\label{eq:moms}
\E(W_N^s)=\sigma^s\cdot
\sum_{r=0}^{s}\Stir{s}{r}(-1)^{s-r}\E(\auffak{(W_N/\sigma)}{r}).
\end{equation}
Next we determine the asymptotics of the rising factorials, similar to the asymptotics of $g_N$.
Neglecting the factor $\auffak{(w_0/\sigma)}{s}$, we obtain for $N=np+k$
\begin{align*}
\prod_{j=0}^{N-1}\frac{T_{j}+s\sigma}{T_{j}}&
=\bigg(\prod_{j=0}^{n-1}\prod_{r=0}^{p-1}\frac{j(p + \frac{\ell}{\sigma})+r +\frac{w_0+b_0}{\sigma}+s}{j(p + \frac{\ell}{\sigma})+r +\frac{w_0+b_0}{\sigma}}\bigg)
\cdot \prod_{r=0}^{k-1}\frac{n(p + \frac{\ell}{\sigma})+r +\frac{w_0+b_0}{\sigma}+s}{n(p + \frac{\ell}{\sigma})+r +\frac{w_0+b_0}{\sigma}}.
\end{align*}
We write again $\psi=(p + \frac{\ell}{\sigma})$ and write the products in terms of Gamma functions.
\begin{align*}
\prod_{j=0}^{N-1}\frac{T_{j}+s\sigma}{T_{j}}&
=\bigg(\prod_{r=0}^{p-1}\frac{\Gamma(n+\frac{r}\psi +\frac{w_0+b_0}{\psi\sigma}+\frac{s}{\psi})\Gamma(\frac{r}\psi +\frac{w_0+b_0}{\psi\sigma})}{\Gamma(\frac{r}\psi +\frac{w_0+b_0}{\psi\sigma}+\frac{s}{\psi})\Gamma(n+\frac{r}\psi +\frac{w_0+b_0}{\psi\sigma})}\bigg)
\cdot \prod_{r=0}^{k-1}\frac{n+\frac{r}\psi +\frac{w_0+b_0}{\psi\sigma}+\frac{s}{\psi}}{n+\frac{r}\psi +\frac{w_0+b_0}{\psi\sigma}}.
\end{align*}
Stirling's formula~\eqref{eq:stirling}
leads to
\begin{align*}
\E(\auffak{(W_N/\sigma)}{s})&=\auffak{(w_0/\sigma)}{s}\cdot n^{s\Lambda}\cdot
\prod_{r=0}^{p-1}\frac{\Gamma(\frac{r}\psi +\frac{w_0+b_0}{\psi\sigma})}{\Gamma(\frac{r}\psi +\frac{w_0+b_0}{\psi\sigma}+\frac{s}{\psi})}
\end{align*}
Thus, by~\eqref{eq:moms} we obtain
\begin{equation}
\label{eq:momConv}
\frac{1}{n^{s\Lambda}}\E(W_N^s)\sim \mu_s= \sigma^s \frac{\Gamma(s+\frac{w_0}\sigma)}{\Gamma(\frac{w_0}\sigma)}
\prod_{r=0}^{p-1}\frac{\Gamma(\frac{r}\psi +\frac{w_0+b_0}{\psi\sigma})}{\Gamma(\frac{r}\psi +\frac{w_0+b_0}{\psi\sigma}+\frac{s}{\psi})}.
\end{equation}
The existence of a distribution, uniquely described by its moment sequence $(\mu_s)$, is readily checked by using
Carleman's criterion~\cite[pp. 189--220]{Carleman23}. We apply Stirling's formula~\eqref{eq:stirling} to the moment sequence, leading to
\[
\mu_s^{-\frac{1}{2s}}\sim \psi^{\frac{p}{2\psi}}\Big(\frac{e}{s}\Big)^{\frac{1-p/\psi}2},
\]
as $s$ tends to infinity. We obtain further
\begin{equation}
\sum_{s=0}^{\infty}\frac{1}{\mu_s^{1/(2s)}}=
\sum_{s=0}^{\infty}\psi^{\frac{p}{2\psi}}\Big(\frac{e}{s}\Big)^{\frac{1-p/\psi}2}=\infty,
\label{eq:Carleman}
\end{equation}
so Carleman's criterion is satisfied, as $1-\frac{p}{\psi}<1$. The convergence of $W_N/N^{\Lambda}$ to 
a limit law follows by the Fr\'echet-Shohat moment convergence theorem~\cite{FrSh1931}, as by~\eqref{eq:momConv} 
as integer moments converge. 

\smallskip

Next, we obtain the density function $f(x)$ of the limit law.  
Recall that by analytic continuation the reciprocal gamma function $1/\Gamma(z)$ is an entire function, with zeros at $z\in\N_0$.  
The gamma function $\Gamma(z)$ is never zero and its only singularities are simple poles at the negative integers.
Therefore $\mu_s$, considered as a complex function of $s$, has simple poles at 
$\rho_j=-\frac{w_0}\sigma-j$, $j=0,1,2\dots$; Then, as $\mu_{s-1}$ is the Mellin transform of the density $f(x)$ of $\mW{\infty}/\kappa$,
an inverse Mellin computation\footnote{See~\cite[Appendix B.7]{FlaSe2009} for more details on this method; see also~\cite[Theorem~6.9]{Jan2010} and \cite[Equation~(6.12)]{Jan2010} for similar results on the class of functions with moments of Gamma type.}
implies that $f(x)$ is expressible in terms of the residues of $\mu_s$: 
\begin{equation*}
f(x)=\frac{1}{2 \pi i} \int_{\sigma-i\infty}^{\sigma+i\infty} \mu_{s-1} x^{-s} ds = 
\sum_{\substack{\text{$\sigma_j=1+\rho_j$ pole of $\mu_{s-1}$}\\ \sigma_j<\sigma}} \Res_{s=\sigma_j}(\mu_{s-1})x^{-\sigma_j},
\end{equation*}
which is valid for $x>0$, and where $\sigma=1+\rho_0+\epsilon$ is 
in the fundamental strip of $ \mu_{s-1}$, and 
\begin{align*}
\Res_{s=\sigma_j}(\mu_{s-1})
&=\Res_{s=\rho_j}(\mu_s)
= \frac{(-1)^j}{j!} \cdot\frac{1}{\Gamma(\frac{w_0}\sigma)}\prod_{r=0}^{p-1}\frac{\Gamma(\frac{r}\psi +\frac{w_0+b_0}{\psi\sigma})}{\Gamma(\frac{r}\psi +\frac{b_0}{\psi\sigma}-\frac{j}{\psi})}.
\end{align*}
Summing for $j\geq 0$ gives the stated density function. 

\smallskip

Next, we provide a representation of th limit law. We determine a product formula for the limit law 
based on its moments in the special case when $\ell/\sigma\in\N$, such that $\psi=p+\frac{\ell}\sigma\in\N$.
First, we convert the moments into a different structure, the extended full product (compare with the discussion in~\cite{BMW2020}) using the multiplicative formula for the Gamma function:
\[
\prod_{k=0}^{m-1}\Gamma\left(z + \frac{k}{m}\right) = (2 \pi)^{\frac{m-1}{2}} \; m^{\frac12 - mz} \; \Gamma(mz).
\]
We use
\[
\prod_{r=0}^{p-1}\frac{\Gamma(\frac{r}\psi +\frac{w_0+b_0}{\psi\sigma})}{\Gamma(\frac{r}\psi +\frac{w_0+b_0}{\psi\sigma}+\frac{s}{\psi})}
=\bigg(\prod_{r=0}^{\psi-1}\frac{\Gamma(\frac{r}\psi +\frac{w_0+b_0}{\psi\sigma})}{\Gamma(\frac{r}\psi +\frac{w_0+b_0}{\psi\sigma}+\frac{s}{\psi})}\bigg)
\cdot 
\prod_{r=p}^{\psi-1}\frac{\Gamma(\frac{r}\psi +\frac{w_0+b_0}{\psi\sigma}+\frac{s}{\psi})}{\Gamma(\frac{r}\psi +\frac{w_0+b_0}{\psi\sigma})}.
\]
Let 
\[
z_1=\frac{w_0+b_0}{\psi\sigma},\quad 
z_2=\frac{w_0+b_0}{\psi\sigma}+\frac{s}{\psi}.
\] 
We have
\begin{align*}
\prod_{r=0}^{\psi-1}\Gamma(\frac{r}{p+ \frac{\ell}\sigma} +z_j)
&=(2 \pi)^{\frac{\psi-1}{2}} \; \psi^{\frac12 - \psi z_j} \; \Gamma(\psi z_j)
, \quad 1\le j\le 2.
\end{align*} This implies that
\begin{align*}
\mu_s&=  \sigma^s \frac{\Gamma(s+\frac{w_0}\sigma)}{\Gamma(\frac{w_0}\sigma)}
\cdot \frac{\Gamma(w_0+b_0)}{\Gamma(w_0+b_0+s \sigma)}
\psi^{s\psi} \prod_{r=p}^{\psi-1}\frac{\Gamma(\frac{r}\psi +\frac{w_0+b_0}{\psi\sigma}+\frac{s}{\psi})}{\Gamma(\frac{r}\psi +\frac{w_0+b_0}{\psi\sigma})}.
\end{align*}
The factor 
\[
\frac{\Gamma(s+\frac{w_0}\sigma)\Gamma(\frac{w_0+b_0}\sigma)}{\Gamma(\frac{w_0}\sigma)\Gamma(\frac{w_0+b_0}\sigma+s)}
\] 
is readily identified as a Beta-distribution.
Since the product of independent random variables satisfy $\E\big( (X_1\cdot X_2)^s\big)=\E(X_1^s)\E(X_2^s)$,
the decomposition into an independent Beta law and generalized Gamma distributions 
with parameters $a=\frac{w_0+b_0+r\sigma}{\sigma}$ and $b=\psi$ follows, leading to the stated result.

\smallskip 

Finally, we turn to the general case and we use the random variable $T$ of Lemma~\ref{lem:PanKu2}:
\[
\E(T^s)=\Gamma(s+1)\cdot\prod_{j=1}^{m}\frac{\Gamma(\frac{j\alpha}{m})}{\Gamma(\frac{(s+j)\alpha}{m})},\quad s\ge 1.
\]
We set $\alpha=p/\psi=p/(p+\ell/\sigma)$, where we readily observe that $0<\alpha<1$.
Next we use $m=p$, such that $\beta=\alpha/p=\frac1\psi$, and obtain 
\[
\E(T^s)=\Gamma(s+1)\cdot\prod_{j=1}^{p}\frac{\Gamma(\frac{j}{\psi})}{\Gamma(\frac{(s+j)}{\psi})},\quad s\ge 1.
\]
We tilt the moment sequence by $\frac{w_0+b_0}{\sigma}-1$ to obtain the sequence
\[
\frac{\Gamma(s+\frac{w_0+b_0}{\sigma})}{\Gamma(\frac{w_0+b_0}{\sigma})}\cdot\prod_{j=1}^{p}\frac{\Gamma(\frac{j-1}{\psi} 
+\frac{w_0+b_0}{\sigma\psi})}{\Gamma(\frac{(s+j-1)}{\psi} 
+\frac{w_0+b_0}{\sigma\psi})},\quad s\ge 1.
\]
Finally, in order to obtain the sequence $(\mu_s)$, we use a Beta distribution with parameters $\frac{w_0}{\sigma}$ and $\frac{b_0}{\sigma}$and Beta-Gamma algebra, noting that
Beta acts as a shift on the moments of a Gamma distribution. At last, we note that by~\eqref{eq:tilt}
\[
\tilt_{\frac{w_0+b_0}{\sigma}-1}(T)=\tilt_{\frac{w_0+b_0}{\sigma}-1}(\tilt_1(L))
=\tilt_{\frac{w_0+b_0}{\sigma}}(L),
\]
leading to the stated result.
\end{proof}

\subsection{Second order asymptotics: martingale tail sums}
Given a martingale $(\mw)$ converging almost surely to a random variable $\mW{\infty}$. The sequence $(\mw - \mW{\infty})$, 
quantifying the difference between the random variable and its almost sure limit is called a martingale tail sum. A classical result for urn models is the central limit theorem and the law of the iterated logarithm by Heyde~\cite{Heyde1977} for the martingale tail sum in the original \Polya\ urn model; see also~\cite{hallheyde80}. Also, a central limit theorem for triangular urns is in the functional limit theorems in~\cite{Gouet93}. The following proposition on martingale tail sums is essentially a restatement of the results of Heyde \cite{Heyde1977} (compare also with~\cite{KuSu2017}).

\begin{prop} \label{prop:heyde}
Let $Z_n, n \geq 0,$ be a zero-mean, $L_2$-bounded martingale with respect to a filtration $\mathcal{G}_n, n \geq 0$. Let $X_n = Z_n - Z_{n-1}, n \geq 1, X_0 := 0,$ and $s_n^2 = \sum_{i=n}^\infty \E(X_i^2)$. Denote $Z$ the almost sure limit of $Z_n$. Assume that, for some non-zero and finite random variable $\eta$, we have, almost surely,
\begin{align} \label{con1}
s_n^{-2} \sum_{i = n}^\infty \E(X_i^2 | \mathcal{G}_{i-1}) \to \eta^2, 
\end{align}
and, for all $\varepsilon > 0$, 
\begin{align} \label{con2} 
s_n^{-2} \sum_{i=n}^\infty  \E(X_i^2 \I{|X_i| \geq \varepsilon s_n}) \to 0. 
\end{align}
Then, 
\begin{align} \label{conv1} 
(\eta s_n)^{-1} (Z_n - Z) \to \mathcal {N}, \quad  s_n^{-1} (Z_n - Z) \to \eta' \mathcal {N},
\end{align} in distribution, where $\eta'$ and $\mathcal N$ are independent and $\eta'$ is distributed like $\eta$. 
If
\begin{itemize}
\item [\textbf{L1.}] $\sum_{i=1}^\infty s_i^{-1} \E[|X_i| \I{|X_i| \geq \varepsilon s_i}] < \infty$ for all $\varepsilon > 0$,
\item [\textbf{L2.}] $\sum_{i=1}^\infty s_i^{-4} \E[X_i^4] < \infty$, 
\end{itemize}
then, almost surely,
\begin{align*}
\limsup_{n \to \infty}  \frac{Z_n - Z}{\eta s_n  \sqrt{2 \log \log s_n^{-1}}}  = 1, \quad \liminf_{n \to \infty} \frac{Z_n - Z}{\eta s_n  \sqrt{2 \log \log s_n^{-1}}}  = -1. 
\end{align*}
\end{prop}

\smallskip 

In the following we study the fluctuations of the scaled number of white balls $\mw=g_N W_N$ in P\'olya-Young urns around their almost sure limit $\mW{\infty}$ and provide a central limit theorem. In view of Proposition~\ref{prop:heyde} we set 
\begin{equation}
\label{def:ZX}
Z_N=\mw-\mW{0}, \quad X_N=Z_N-Z_{N-1}=\mw-\mW{N-1},
\end{equation} 
and study the properties of $X_N$. 
In the following we collect the asymptotic expansions of the conditional and unconditional
expected value of the second moment $X_N^2$.

\begin{lem}
\label{lem:secondMom}
The second moment of $X_N=\mw-\mW{N-1}$ satisfies 
\[
\E(X_{N+1}^2\given \field_N) \sim N^{-\Lambda-1} \sigma^2 \kappa  \cdot \mW{\infty},
\quad \E(X_{N+1}^2)\sim N^{-\Lambda-1} \sigma^2 \kappa  \cdot W_0.
\]
and
\[
s_N^2 \sim N^{-\Lambda} \frac{\sigma^2}{\Lambda} \kappa  \cdot W_0.
\]
\end{lem}
\begin{proof}
For the sake of convenience, we study $X_{N+1}=\mW{N+1}-\mw$. We obtain
\begin{equation}
X_{N+1}=g_{N+1}\Big(\big(\frac{1}{g_{N}}-\frac1{g_{N+1}}\big)\mw + \Delta_{N+1}\Big).
\label{eq:XDef}
\end{equation}
By the dynamics of the urn we get
\begin{align*}
\E(X_{N+1}^2\given \field_N) &=g_{N+1}^2\Big(\big(\frac{1}{g_{N}}-\frac1{g_{N+1}}\big)\mw\Big)^2\cdot \frac{B_N}{T_N}\\
&+ g_{N+1}^2\Big(\big(\frac{1}{g_{N}}-\frac1{g_{N+1}}\big)\mw+\sigma\Big)^2\frac{W_N}{T_N}.
\end{align*}
We note that $B_N=T_N-W_{N}$, as well as $W_N=\mw/g_N$. Simplification gives
\begin{align}
\label{eqn:exactSecMom}
\E(X_{N+1}^2\given \field_N) &=g_{N+1}^2\big(\frac{1}{g_{N}}-\frac1{g_{N+1}}\big)^2\mw^2\\
&\quad+ 2\sigma \frac{g_{N+1}^2}{g_N}\big(\frac{1}{g_{N}}-\frac1{g_{N+1}}\big)\cdot \frac{\mw^2}{T_N}+\frac{g_{N+1}^2}{g_N}\sigma^2\frac{\mw}{T_N}.
\end{align}
Concerning asymptotics, we know that $g_N$ is of order $N^{-\Lambda}$~\eqref{eq:Def_gN}, and we have
\begin{equation}
\label{def:DiffGn}
\big(\frac{1}{g_{N}}-\frac1{g_{N+1}}\big)=\mathcal{O}\big(n^{\Lambda-1}\big).
\end{equation}
Consequently, we obtain
\[
\E(X_{N+1}^2\given \field_N) \sim N^{-\Lambda-1} \sigma^2 \kappa  \cdot \mW{\infty}.
\]
For the expected value we turn again to~\eqref{eqn:exactSecMom} and take expectations. By the tower property of the expected value and the martingale property for $\mw$, such that 
\[
\E(\mw)=\E(\mW{0})= g_0 W_0=W_0,
\]
we obtain
\[
\E(X_{N+1}^2)\sim N^{-\Lambda-1} \sigma^2 \kappa  \cdot W_0.
\]
\end{proof}
Next, we study the fourth moment of $X_N$.
\begin{lem}
\label{lem:fourthMom}
The fourth moment of $X_N=\mw-\mW{N-1}$ satisfies 
\[
\E(X_{N+1}^4)\sim N^{-3\Lambda-1}\kappa^3 \sigma^4 W_0.
\]
\end{lem}
\begin{proof}
From~\eqref{eq:XDef} we get
\begin{align*}
\E(X_{N+1}^4\given \field_N) &=g_{N+1}^4\Big(\big(\frac{1}{g_{N}}-\frac1{g_{N+1}}\big)\mw\Big)^4\cdot \frac{T_N-W_N}{T_N}\\
&+ g_{N+1}^4\Big(\big(\frac{1}{g_{N}}-\frac1{g_{N+1}}\big)\mw+\sigma\Big)^4\frac{W_N}{T_N}.
\end{align*}
We simplify using the binomial theorem and apply the expansions~\eqref{eq:Def_gN} and\eqref{def:DiffGn}. 
This leads to 
\begin{align*}
\E(X_{N+1}^4\given \field_N) &=g_{N+1}^4\Big(\big(\frac{1}{g_{N}}-\frac1{g_{N+1}}\big)\mw\Big)^4\\
\frac{g_{N+1}^4}{g_N}\frac{\mw}{T_N}\sigma^4 + \mathcal{O}(N^{-2\Lambda-2}),
\end{align*}
which leads the stated asymptotics of the expected value using the expansion~\eqref{eq:Def_gN}.
\end{proof}

We state the main result of this section, a central limit theorem for the martingale tail sum of the number of white balls. 
\begin{theorem}
Let $W_N$ be the number of white balls at time $N$ in a generalized P\'olya-Young urn. Then, $\mw = g_N W_N$ is an almost surely
convergent martingale. In distribution, we obtain the central limit theorem for the martingale tail sum $\mw-\mW{\infty}$:
\[
N^{\frac{\Lambda}2} \frac{\beta}\eta(\mw-\mW{\infty}) \to \mathcal{N},
\]
with $\beta_N=\frac{\sqrt{\Lambda}}{\sigma\sqrt{\kappa}}$ and $\eta=\sqrt{\mW{\infty}}$. Moreover, it holds almost surely that
\begin{align*}
\limsup_{N \to \infty}  \frac{\mw - \mW{\infty}}{\eta  N^{\frac{\Lambda}2} \beta  \sqrt{2 \log \log N^{\frac{\Lambda}2} \beta}}  = 1, \quad \liminf_{N \to \infty} \frac{\mW{\infty}}{\eta N^{\frac{\Lambda}2} \beta  \sqrt{2 \log \log N^{\frac{\Lambda}2} \beta}}  = -1. 
\end{align*}

\end{theorem}
\begin{proof}
We begin to check condition~\eqref{con2} and use Lemma~\ref{lem:fourthMom}:
\begin{align}
\E(X_i^2 \I{|X_i| \geq \varepsilon s_N}) &=
\int_{\Omega} X_i^2 (\I{|X_i| \geq \varepsilon s_N})d\P
=\int_{|X_i| \geq \varepsilon s_N}X_i^2 d\P
\le \int_{|X_i| \geq \varepsilon s_N}\frac{X_i^4}{\varepsilon^2 s_N^2} d\P
\\
&\le \frac1{\varepsilon^2 s_N^2}\int_{\Omega }X_i^4 d\P
= \frac1{\varepsilon^2 s_N^2}\E(X_i^4)
\sim \frac1{\varepsilon^2} N^{-2\Lambda-1}\kappa^3 \frac{\sigma^2\kappa}{\Lambda},
\end{align}
for $i\ge N$. This implies that
\[
\frac1{s_N^2}\sum_{i=N}^\infty  \E(X_i^2 \I{|X_i| \geq \varepsilon s_N}) 
\] 
is of order $N^{-2\Lambda}\cdot N^{\Lambda}=N^{-\Lambda}$ and tends to zero for $N\to\infty$. 
Concerning condition~\eqref{con1} we use Lemma~\ref{lem:secondMom} and obtain
\[
s_N^2 \sum_{i=N}^\infty \E(X_{N+1}^2\given \field_N)
\to   W_0 \mW{\infty}=\eta^2.
\]
Next we check the conditions \textbf{L1.} and \textbf{L2.}.
For \textbf{L2.} we combine Lemmata~\ref{lem:secondMom} and~\ref{lem:fourthMom}. 
As
\[
s_N^{-4} \E(X_N^4)=\mathcal{O}( N^{-\Lambda-1}),
\]
the sum is convergent. 
For \textbf{L1.} we proceed similarly to the proof of condition~\eqref{con2}:
\begin{align*}
\E(|X_i| \I{|X_i| \geq \varepsilon s_i})
&=\int_{\Omega} |X_i| (\I{|X_i| \geq \varepsilon s_i})d\P
\le \int_{\Omega}\frac{|X_i|^4}{\varepsilon^3 s_i^3} d\P
&= \frac1{\varepsilon^3 s_i^3}\E(X_i^4).
\end{align*}
By the asymptotics of $s_N$ and $\E(X_N^4)$ from Lemmata~\ref{lem:secondMom} and~\ref{lem:fourthMom}
we note that $\frac{\E(X_i^4)}{s_i^4}$ is of order 
$i^{-1-3\Lambda}\cdot i^{2\Lambda}=i^{-1-\Lambda}$, proving \textbf{L1.}
\end{proof}

\section{P\'olya-Young urns with multiple colors}
\begin{defi}[Multicolor P\'olya-Young urns]
For period $p$, balance $\sigma$ and
parameter $\ell\ge 0$ we introduce P\'olya-Young urns with $t\ge 2$ different types (colors):
\[
M_1=\dots=M_{p-1}=
\text{diag}(\sigma,\dots,\sigma),
\quad
M_j=
\left(
\begin{matrix}
\sigma&0\dots &0& \ell\\
0&\sigma&0\dots & \ell\\
 \vdots &&&\vdots\\
0&\dots&\dots  0& \sigma+\ell\\
\end{matrix}
\right),\quad 1\le j\le p.
\]
\end{defi}
Let $W_{N,\ell}$ denote the number of balls of type $\ell$, $1\le \ell\le t$ at time $N$, 
with $T_N= \sum_{\ell=1}^{t}W_{N,\ell}$. We are interested in the distribution of the random vector 
\[
\Wn=(W_{N,1},\dots,W_{N,t-1})
\]
both for 
fixed $N\ge 1$, as well as in the limit $N\to\infty$. 

\subsection{Explicit Moments and limit law} 
In our to obtain the mixed moments of $\Wn$ we use again martingales. 
Given positive integers $s_1,\dots,s_{t-1}$, we use the shorthand notation $\Wns$ for the product
\[
\Wns=\prod_{\ell=1}^{t-1}\binom{\frac{W_{N,\ell}}{\sigma}+s_\ell-1}{s_\ell}
=\prod_{\ell=1}^{t-1}\frac{\auffak{(W_{N,\ell}/\sigma)}{s_\ell}}{s_\ell!}.
\]
\begin{lem}
\label{lem:MultiMartin}
Let $g_{N,\mathbf{s}}=\prod_{k=0}^{N-1}\big(\frac{T_k}{T_k+\sigma\sum_{\ell=1}^{t-1}s_\ell}\big)$. The random variable 
$g_{N,\mathbf{s}}\Wns$ is a non-negative martingale. 
\end{lem}
\begin{proof}
By definition of the urn process, only one of the random variables can increase by $\sigma$ after a draw.
Consequently, only a single random variable $\frac{W_{N,\ell}}{\sigma}$ may increase by one.
We obtain
\begin{align*}
\E\bigl(\Wns \given \field_{N-1}\bigr)&= 
\Wnms + \sum_{\ell=1}^{t-1} \binom{\frac{W_{N,\ell}}{\sigma}+s_\ell-1}{s_\ell-1}\frac{ W_{N-1,\ell}}{T_{N-1}}\prod_{\substack{j=1\\ j\neq \ell }}^{t-1} \binom{\frac{W_{N,\ell}}{\sigma}+s_j-1}{s_j}
\\
&=\Wnms + \sum_{\ell=1}^{t-1}\binom{\frac{W_{N,\ell}}{\sigma}+s_\ell-1}{s_\ell-1} \frac{\Wnms}{\binom{\frac{W_{N,\ell}}{\sigma}+s_\ell-1}{s_\ell}}
\cdot \sigma\cdot\frac{W_{N-1,\ell}}{T_{N-1}}\\
&=\Wnms \Big(1+\frac{\sigma\sum_{\ell=1}^{t-1}s_\ell}{T_{N-1}}\Big).
\end{align*}
This directly leads to the stated martingale structure. 
\end{proof}
From the martingale structure we directly obtain the mixed rising factorial moments and then the limit law.
\begin{theorem}
\label{the:MV}
The mixed rising factorial moments of $W_{N,1},\dots,W_{N,t-1}$ are given by
\[
\E\big(\prod_{\ell=1}^{t-1}\auffak{(W_{N,\ell}/\sigma)}{s_\ell}\big) =
\frac{1}{g_{N,\mathbf{s}}}\prod_{\ell=1}^{t-1}\frac{\Gamma(\frac{w_{0,\ell}}{\sigma}+s_\ell)}{\Gamma(\frac{w_{0,\ell}}{\sigma})}. 
\]
The scaled random variables $W_{N,1}/(\sigma N^{\Lambda}),\dots,W_{N,t-1}/(\sigma N^{\Lambda})$ jointly converge
to a limit law, determined by its moment sequence
\[
\mu_{s_1,\dots,s_{t-1}}=
\bigg(\prod_{\ell=1}^{t-1}\frac{\Gamma(\frac{w_{0,\ell}}{\sigma}+s_\ell)}{\Gamma(\frac{w_{0,\ell}}{\sigma})}\bigg)
\prod_{r=0}^{p-1}\frac{\Gamma(\frac{r}\psi +\frac{\sum_{\ell=1}^{t}w_{0,\ell}}{\psi\sigma})}{\Gamma(\frac{r}\psi +\frac{\sum_{\ell=1}^{t}w_{0,\ell}}{\psi\sigma}+\frac{s_1+\dots+s_{t-1}}{\psi})}.
\]
For $\ell/\sigma\in\N$, the distribution of the limit law is given by a product of a Dirichlet distribution and independent generalized Gamma distributions, 
\[
(D_1\cdot G,\dots,D_{t-1}\cdot G),\quad 
\text{where } (D_1,\dots,D_t)=\Dir(\frac{w_{0,1}}\sigma,\dots,\frac{w_{0,t}}\sigma)
\]
and $G$ denotes a product of generalized Gamma distributions 
\[
G=\prod_{\ell=p}^{\psi-1}GG(\frac{\sum_{\ell=1}^{t}w_{0,\ell}+\ell\sigma}{\sigma},\psi).
\]
In the general case, the distribution of the limit law is given by a Dirichlet distribution and the scaled local time $L$ of a noise-reinforced Bessel process of dimension $d=\frac{2\ell}{p\sigma +\ell}$ and reinforcement parameter $r=-\frac{p-1}2$:
\[
(D_1\cdot G,\dots,D_{t-1}\cdot G),
\]
where 
\[
(D_1,\dots,D_t)=\Dir(\frac{w_{0,1}}\sigma,\dots,\frac{w_{0,t}}\sigma),
\quad\text{and } G=\tilt_{\frac{\sum_{\ell=1}^{t}w_{0,\ell}}{\sigma}}(L).
\]
\end{theorem}

\begin{proof}
Our proof strategy follows closely the proofs of Theorems~\ref{the1} and~\ref{the2}. 
First, we observe the obvious fact that the total number of balls~\eqref{eq:TotalPoYou}
has the form
\begin{equation}
\label{eq:TotalPoYouMult}
T_N=\big(n(p-1)+k\big)\sigma +n(\sigma+\ell)+\sum_{j=1}^{t}w_{0,j}
=n(p\sigma + \ell)+k\sigma +\sum_{j=1}^{t}w_{0,j}.
\end{equation}

From Lemma~\ref{lem:MultiMartin} and the tower property of expectations we readily obtain
a recurrence relation for $\E(\Wns)$:
\[
\E(\Wns)=\frac{T_{N-1}+\sigma\sum_{\ell=1}^{t-1}s_\ell}{T_{N-1}}\E(\Wnms),\quad N\ge 1.
\]
From this recurrence relation we directly get an explicit expression for $\E(\Wns)$,
\[
\E(\Wns)=\frac{1}{g_{N,\mathbf{s}}}\mathbf{W}_{0,s}.
\]
Next, we translate the result to rising factorials, obtaining the stated formula for the moments.
By~\eqref{eq:risingToRaw} the raw mixed moments are expressed in terms of the rising factorial moments.   
An application of Stirling's formula~\eqref{eq:stirling} leads then to moment convergence:
\[
\E\big(W_{N,1}^{s_1}/(\sigma N^{\Lambda})^{s_1},\dots,W_{N,t-1}^{s_{t-1}}/(\sigma N^{\Lambda})^{s_{t-1}}\big)
\to \mu_{s_1,\dots,s_{t-1}},
\]
for $s_1,\dots,s_{t-1}\ge 0$.
In order to identify the moments we first note that the appearance of the sum $\sum_{j=1}^{t-1}s_j$
indicates a common factor in the limit law. Following the proof of Theorem~\ref{the2} step by step this factor leads either
to a product of generalized Gamma distributed random variables or to the local time of a Bessel process. 
Finally, we look at the moments of a Dirichlet distribution
$(D_1,\dots,D_t)=\Dir(\frac{w_{0,1}}\sigma,\dots,\frac{w_{0,t}}\sigma)$, 
which are given by
\[
\E\big(D_1^{s_1}\cdot\dots\cdot D_{t-1}^{s_{t-1}}\big)
=\frac{\Gamma(\sum_{j=1}^{t}\frac{w_{0,j}}{\sigma})}{
\Gamma(\frac{w_{0,t}}{\sigma}+\sum_{j=1}^{t-1}(\frac{w_{0,j}}{\sigma}+s_{j}))}\cdot
\prod_{i=1}^{t-1}\frac{\Gamma(s_i+\frac{w_{0,i}}\sigma)}{\Gamma(\frac{w_{0,i}}\sigma)}.
\]
In the case of $p+\ell/\sigma\in\N$ we directly observe the moments of the Dirichlet distribution, whereas in the general case 
we simply multiply with
\[
1=\frac{\Gamma(\sum_{j=1}^{t}\frac{w_{0,j}}{\sigma})}{
\Gamma(\frac{w_{0,t}}{\sigma}+\sum_{j=1}^{t-1}(\frac{w_{0,j}}{\sigma}+s_{j}))}\cdot 
\frac{
\Gamma(\frac{w_{0,t}}{\sigma}+\sum_{j=1}^{t-1}(\frac{w_{0,j}}{\sigma}+s_{j}))}{\Gamma(\sum_{j=1}^{t}\frac{w_{0,j}}{\sigma})}
\]
to observe the stated decomposition.
\end{proof}

\section{Increasing tree families with periodic immigration}\label{sec:inc}
\subsection{Introduction to increasing tree families}
Increasing trees, also called increasingly labelled trees, are rooted labelled trees. The nodes of a tree $T$ of size $|T| = n$, with integer $n\ge 1$, are labeled with distinct integers from a label set $\mathcal{M}$ of size $|\mathcal{M}|=n$. Here, the size $|T|$ of a tree denotes the number of vertices of $T$ (and thus coincides with the number of labels). One chooses as label set the first $n$ positive integers, i.e., $\mathcal{M} = [n] := \{1, 2, \dots, n\}$, in such a way that the label of any node in the tree is smaller than the labels of its children. As a consequence, the labels of each path from the root to an arbitrary node in the tree are forming an increasing sequence, which explains the name of such a labelling. Various increasing tree models turned out to be appropriate in order to describe the growth behavior of quantities in a number of applications and occurred in the probabilistic literature. 

\begin{figure}[!htb]
\centering
\includegraphics[scale=0.7]{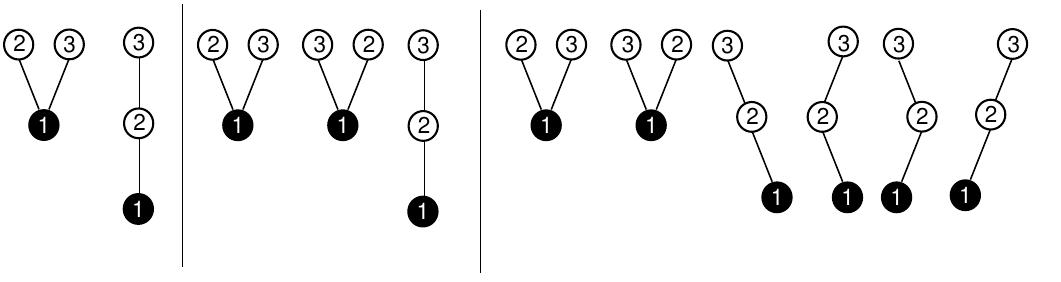}
\caption{Two recursive trees of size three (no left-to-right order); three plane recursive trees of size three; six binary increasing trees.}
\label{fig:IncTrees}
\end{figure}

E.g., they are used to describe the spread of epidemics, to model pyramid schemes, and as a simplified growth model of the world wide web. See Mahmoud and Smythe~\cite{MahSmy1995-2} for a survey collecting results about recursive trees, a subfamily of increasing trees, prior 1995. For general results about increasing trees, especially their enumeration, we refer to~\cite{BerFlaSal1992,PanPro2007},
as well as the book of Drmota~\cite{Drmota} and references therein. Increasing trees are also intimately connected to urn models, see for example~\cite{MahSmy1991} or the recent works~\cite{AddarioA,Janson2005RSA}. Next we are going to describe in more detail the tree evolution process which generates random trees (of arbitrary size $n$) of grown simple families of increasing trees. This description is a consequence of the considerations made by Panholzer and Prodinger~\cite{PanPro2007}. We also refer to their work, as well as~\cite{BerFlaSal1992,Drmota}, for more information about combinatorial properties of increasing tree families, especially enumerative results.
\begin{itemize}
\item Step $1$: The process starts with the root labelled by $1$.
\item Step $i+1$: At step $i+1$ the node with label $i+1$ is attached to any previous node $v$ with out-degree $d(v)$) of the already grown tree of size $i$ with probabilities 
given by
\begin{equation*}
   p(v)  =
   \begin{cases}
      \displaystyle{\frac{1}{i}}, & \quad \text{for recursive trees (A)}, \\
      \displaystyle{\frac{d-d(v)}{(d-1)i+1}}, & \quad \text{for $d$-ary increasing trees (B) }, \\
      \displaystyle{\frac{d(v)+\alpha}{(\alpha+1)i-1}}, & \quad \text{for GPORTs (C)};
   \end{cases}
\end{equation*}
\end{itemize}
here the abbreviation GPORTs stands for generalized plane-oriented recursive trees; the special case $\alpha=1$ leads to ordinary plane-oriented recursive trees. We note in passing that the naming of these trees is not uniform: recursive trees are sometimes called increasing Cayley trees, and (generalized) plane-oriented recursive trees are also known as~\cite{AddarioA,BA,MPP,Prod1996,JerzyA} heap-ordered trees, scale-free trees\footnote{Albeit with a slight different initial configuration}, preferential attachment trees, and B\'arabasi-Albert trees. 
\begin{figure}[!htb]
\centering
\includegraphics[scale=0.7]{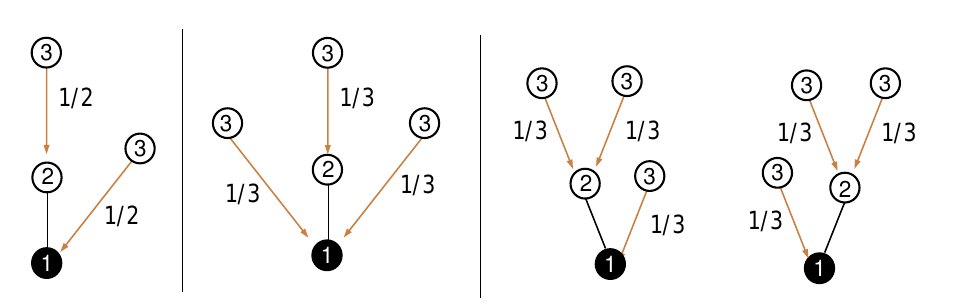}
\caption{Generating recursive trees, plane recursive trees and binary increasing trees by insertion of node labeled three.}
\label{fig:IncTrees2}
\end{figure}

\subsection{Periodic immigration}
In this section we generalize the growth processes for classical increasing trees, as stated in the introduction, 
to increasing trees with immigration. We start with a single tree of size $1$. Given a period $p$, $p>1$: we perform ordinary growth steps, entering nodes labeled two, three, etc., according to one of the three families, i.e. recursive trees, $d$-ary increasing trees and generalized plane-oriented recursive trees. When $N=np$ is a multiple of $p$ step, we first add an ordinary node labeled $N$, and then additionally add a new special root node, also labeled $p$. These new root nodes have connectivity $\ell$, with real $\ell>0$ for recursive trees and generalized plane-oriented increasing trees, and $\ell\in\N$ for $d$-ary increasing trees, leading to the creation of an increasing forest. For $N=np+k\ge 1$, with $0\le k<p$, let $C_N$ denote the total connectivity, a quantity which governs the probabilities of connections of nodes. 
\begin{equation}
\label{eqn:Connectivity1}
C_N=\big(n(p-1)+k\big)\sigma +n(\sigma+\ell)-\kappa
=n(p\sigma + \ell)+k\sigma+\kappa,
\end{equation}
where the constants $\sigma$ and $\kappa$ depend on the specific tree family: 
\begin{equation}
\label{eqn:Connectivity2}
\sigma =
\begin{cases}
1, \\
d-1,\\
1+\alpha, \\
\end{cases},
\quad
\kappa= 
\begin{cases}
0, &\text{for recursive trees}, \\
1, &\text{for $d$-ary increasing trees},\\
-1,& \text{for GPORTs}.\\ 
\end{cases}
\end{equation}
We describe the growth processes of increasing trees with periodic immigration in more detail. 
We distinguish between ordinary nodes labeled one, two , three, etc. and special nodes which are created at time steps which are multiples of the period $p$.
These special nodes are new roots, which may attract ordinary nodes. This leads to an increasing forest, with an ordinary root labeled one 
and additional roots created at times $p,2p,3p,\dots$.  We label the new root nodes by their creation time stamp. Thus, their are two nodes labeled $mp$, $m\ge 1$: a root node and an ordinary node. In the following we use the term ordinary nodes for the nodes labeled one up to $N$, including the original root labeled one. 

\begin{defi}
The following process generates increasing tree families with periodic immigration.
\begin{itemize}
\item Step $1$: The process starts with the root labeled by $1$.
\item Step $i+1$: the node with label $i+1$ is attached to any previous ordinary node $v$ with out-degree $d(v)$ in the increasing forest with $i$ ordinary nodes with probabilities 
given by
\begin{equation}
\label{eqn:incTree1}
   p(v)  =
   \begin{cases}
      \displaystyle{\frac{1}{C_i}}, & \quad \text{for recursive trees}, \\[0.1cm]
      \displaystyle{\frac{d-d(v)}{C_i}}, & \quad \text{for $d$-ary increasing trees}, \\[0.1cm]
      \displaystyle{\frac{d(v)+\alpha}{C_i}}, & \quad \text{for GPORTs};\\
   \end{cases}
\end{equation}
Additionally, the node with label $i+1$ is attached to any previously created new root nodes $v$, there are $\big\lfloor \frac{i}{p}\big\rfloor$ many of them, with probability 
given by
\begin{equation*}
   p(v)  =
   \begin{cases}
      \displaystyle{\frac{\ell}{C_i}}, & \quad \text{for recursive trees}, \\[0.1cm]
      \displaystyle{\frac{\ell-d(v)}{C_i}}, & \quad \text{for $d$-ary increasing trees}, \\[0.1cm]
      \displaystyle{\frac{d(v)+\ell}{C_i}}, & \quad \text{for GPORTs};\\
   \end{cases}
\end{equation*}
\item Step $i+1$, with $i+1\mod p=0$: a new root is created.
\end{itemize}
\end{defi}

We point out that we always have to wait $p$ insertion steps for the creation of a new root.
However, as the node labeled one is the root of the original tree, the first new root is actually created $p-1$ steps after the original root. 
Of course, one may readily modify this process to start with the original root labeled zero, leading to a waiting time of $p$ insertion steps.
Later on, such a modified process will turn out to be useful to relate this growth process to the chinese restaurant process.

\begin{figure}[!htb]
\centering
\includegraphics[scale=0.7]{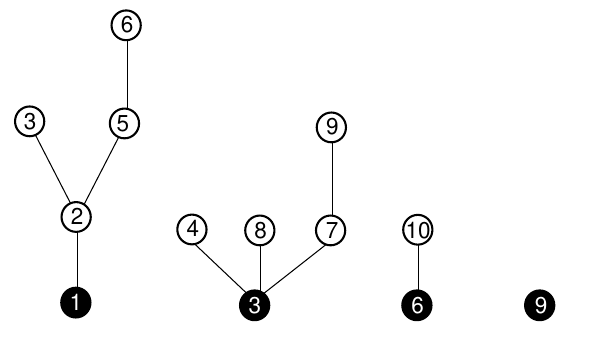}
\caption{An increasing tree of size $10$ with periodic immigration, $p=3$; the three additional new roots have the same labels as the time stamp of their creation.}
\label{fig:IncTreesImmigration}
\end{figure}

\begin{remark}[Restrictions on the immigration]
For recursive trees this generalized growth process is closely related to so-called weighted recursive trees.  Let $N\ge 1$. Starting with a tree of size $N$, we attach node $N+1$ to the vertex labeled $k$ with probability $\frac{w_k}{W_N}$, with total weight $W_N=\sum_{k=1}^{N}w_k$. Such trees are known in the literature (see, e.g., \cite{BorVat2006}) as weighted recursive trees, as for constant $w_k=c$, $k\ge 1$, this reduces to ordinary recursive trees. 
In our case, the weights of ordinary nodes are $w_k=1$, whereas immigrated nodes, the new roots in the increasing forests,
have weight $w_k=\ell$. For $d$-ary increasing trees we assume that $\ell\in\N$. Thus, the forest consists of new trees with $\ell$-ary roots.
It is possible to extend the construction to real $\ell>0$. But then, in order to obtain a well-defined growth process, $p(v)  =\frac{\ell}{C_i}$ for new root nodes $v$, 
and further $\displaystyle{\frac{d-1-d(v)}{C_i}}$ for the children of the new roots.
\end{remark}

\subsection{Descendants in increasing trees with periodic immigration}

It the following we turn to so-called label-based parameters~\cite{KuPaLabel} in increasing tree families and relate them with P\'olya-Young urns. We are interested in the number of descendants of a specific node $j$ (with $1 \le j \le N$), i~e. the size of the subtree rooted at $j$ (where size is measured as usual by the number of nodes), in a random size-$N$ increasing tree with immigration. Note that $j$ is counted as a descendant of itself. This parameter is a classical parameter of interest, see~\cite{KuPaIncDesc,MahSmy1991} and the references therein. 
Concerning classical increasing trees, it is well known that the number of descendants of node $j$,
can be modeled not only by trees, but also by urn models~\cite{KuPa2014,KUBA2023113443,MahSmy1991}. 
We extend this connection to increasing trees with immigration.

\begin{theorem}[P\'olya-Young urns and Descendants of node $j$]
The descendants $D_{N,j}$ of node $j$, $j\ge 1$, in grown increasing tree families with immigration of size $N$, with parameters $p\in\N$ and $\ell>0$, 
are distribution as the number of white balls in the P\'olya-Young urns at time $N-j$, with period $p$ and parameter $\ell>0$:
\[
D_{N,j}\law \frac{W_{N-j}-\kappa}{\sigma},
\]
with constants $\kappa$ and balance $\sigma$ given by~\eqref{eqn:Connectivity2}. 
The initial conditions are given by
\[
w_0= \kappa+\sigma=
\begin{cases}
1, &\text{for recursive trees}, \\
d, &\text{for $d$-ary increasing trees},\\
\alpha,& \text{for GPORTs}\\ 
\end{cases}
\qquad 
b_0=C_j-w_0,
\]
where $C_N$ denotes the connectivity~\eqref{eqn:Connectivity1}.
\end{theorem}

\begin{remark}[Multivariate Connection]
This connections extend to the joint distribution of the descendants of ordinary nodes $j_1,\dots,j_t$, but with a random initial configuration, as we start the urn process when node $j_t$ is present, but the initial configurations depends on the events from time step $j_1+1$ to step $j_t$. Moreover, one has to take care whether a node $j_k$ is contained in another subtree of $j_1,\dots,j_{k-1}$.
\end{remark}

\begin{remark}[Descendants of the roots]
Ordinary increasing trees only possess a only single root, such that $D_{N,1}=N$. In contrast, for increasing trees with periodic immigration, the number of descendants $D_{N,1}$ of the root is not deterministic anymore, as the new root nodes may also attract nodes.
In it also possible to track the individual tree sizes in the forest 
using different initial conditions: for the number of descendants $D^{(r)}_{N,mp}$ of the root created at time $m\cdot p$ the initial condition  is $w_0=\ell$, and the connection to the urn is 
\[
D^{(r)}_{N,mp}\law \frac{W_{N-mp}-w_0}{\sigma}.
\]
We point out that the root nodes satisfy the identity
\[
D_{N,1}+\sum_{m=1}^{\lfloor N /p \rfloor} D^{(r)}_{N,mp} = N
\]
Additionally, we note that the descendants of the roots labeled $1,p,\dots,(t-1)p$ are related to the distribution of the balls of type $1,\dots,t-1$ in multicolor P\'olya-Young urns of dimension $t$: one starts with an urn of size two at time $p$, and increases the number of types each time a new root is created until time $t$, where the time $tp$, where type $t$ encodes the number of nodes not contained in the subtrees rooted at $1,p,\dots,(t-1)p$. 
\end{remark}

\begin{proof}
Following~\cite{KuPa2014}, in order to observe this connection, we observe first that $D_{j,j}=1$ for all three tree families. Next, we proceed similar to~\cite{KuPaIncDescComb}.
Let $\{N+1<_c v\}$ denote the event that the node labeled $N+1$ is attached to the node $v$. 
We analyse the conditional probabilities $\P\{D_{N+1,j}=m+1|D_{N,j}=m\}$. Under our assumption the subtree rooted at the node labeled $j$ has size $m$
and consists of nodes $v_k$, $1\le k\le m$, (with labels larger than $j$).
We obtain
\[
\P\{D_{N+1,j}=m+1|D_{N,j}=m\}= \sum_{k=1}^m \P\{N+1<_c v_k\}.
\] 
For each family the probabilities are determined by~\eqref{eqn:incTree1}. 
The sum of outdegrees in the subtree rooted at label $j$ is deterministic. By our assumption $D_{N,j}=m$ we have
\[
\sum_{k=1}^{m}d(v_k)=m-1.
\]
Consequently, 
\[
\P\{D_{N+1,j}=m+1|D_{N,j}=m\}= 
\begin{cases}
 \displaystyle{\frac{m}{C_N}}, & \quad \text{for recursive trees}, \\[0.1cm]
      \displaystyle{\frac{(d-1)m+1}{C_N}}, & \quad \text{for $d$-ary increasing trees}, \\[0.1cm]
      \displaystyle{\frac{(\alpha+1)m-1}{C_N}}, & \quad \text{for GPORTs};\\
\end{cases}
\]
Moreover, $\P\{D_{N+1,j}=m|D_{N,j}=m\}=1-\P\{D_{N+1,j}=m+1|D_{N,j}=m\}$.
This implies that the stated connection, as the dynamics of the urn model, as described in the introduction, matches the dynamics of the process generating 
increasing trees with periodic immigration.
\end{proof}

\subsection{Outdegree distribution in increasing trees with periodic immigration}

Next, we turn to the distribution of the outdegree of label $j$ in generalized plane-oriented recursive trees~\cite{KuPaLabel}. 
\begin{theorem}[P\'olya-Young urns and Outdegree of node $j$]
\label{theOutdegree}
The oudegree $X_{N,j}$ of node $j$, $j\ge 2$, in an generalized plane-oriented recursive tree of size $n$ with periodic immigration with parameters $p\in\N$ and $\ell>0$, 
is related to the number of white balls $W_{N-j}$ after $N-j$ draws in a modified P\'olya-Young urn with period $p$ and $\ell>0$:
\[
X_{N,j}= W_{N-j}-\alpha,\quad w_0=\alpha,\, b_0=C_j-w_0.
\]
The transition matrices are given by
\[
M_j=
\left(
\begin{matrix}
1&1+\alpha\\
0&1+\alpha\\
\end{matrix}
\right),\quad 1\le j\le p-1,
\quad
M_p=
\left(
\begin{matrix}
1&1+\alpha+\ell\\
0&1+\alpha+\ell\\
\end{matrix}
\right).
\]
\end{theorem}
\begin{proof}
By~\eqref{eqn:incTree1} we have $X_{j,j}=0$ and 
\[
\P\{X_{N+1,j}=m+1|X_{N,j}=m\}= \frac{m+\alpha}{C_N}.
\]
Thus, we readily observe the stated connection to the urn model with initial conditions $w_0=\alpha$ and $b_0=C_j-w_0$.
\end{proof}

This guides us to study the following model of periodic urns. 
\begin{defi}[Periodic triangular urns]
Let $p\in\N$ and $\sigma,\ell_1,\ell_2>0$ denote positive reals. A periodic triangular urn of period $p$, with parameters $\sigma,\ell_1,\ell_2$ 
is determined by the replacement matrices
\[
M_1=M_2=\dots=M_{p-1}=
\left(
\begin{matrix}
\sigma&\ell_1\\
0&\sigma+\ell_1\\
\end{matrix}
\right),
\quad\text{and }
M_{p}=
\left(
\begin{matrix}
\sigma&\ell_2\\
0&\sigma+\ell_2\\
\end{matrix}
\right).
\]
\end{defi}
For $\ell_1=\ell_2$ the model degenerates to an ordinary triangular urn. We emphasize that the total number of balls $T_N$ in a periodic triangular urns urn at time $N=np+k$, with $n\in\N_0$, $0\le k\le p-1$ is given by
\begin{equation}
\label{eq:TotalPerTria}
T_N=w_0+b_0+\big(n(p-1)+k\big)(\sigma+\ell_1) +n(\sigma+\ell_2)
=N\sigma_1 +n(\ell_2-\ell_1)+w_0+b_0,
\end{equation}
where
\begin{equation}
\label{eq:TotalPerTria2}
\sigma_1=\sigma+\ell_1.
\end{equation}

Next we collect the exact results for $W_N$, as well as the martingale structure.
\begin{lem}[Periodic triangular urns: distribution and almost sure limit]
\label{lem:Triangular}
The distribution of $W_N$ is determined by its 
rising factorial moments and its probability generating function, 
given verbatim as in Theorem~\ref{the1}. Let $N=np+k$, $0\le k< p$ and let $n\to\infty$.
\begin{equation}
\label{eq:Def_gNTriangular}
g_N= \prod_{j=0}^{N-1}\frac{T_{j}}{T_{j}+\sigma},\quad
g_N\sim \kappa \cdot N^{-\Lambda} ,\quad \kappa= p^{\Lambda}\prod_{r=0}^{p-1}\frac{\Gamma(\frac{r}\psi  +\frac{w_0+b_0}{\sigma_1\psi}+\frac{\sigma}{\sigma_1\psi})}{\Gamma(\frac{r}\phi +\frac{w_0+b_0}{\sigma_1 \psi})},
\end{equation}
with $T_N$ given by~\eqref{eq:TotalPerTria}, 
\[
\psi=p+\frac{\ell_2-\ell_1}{\sigma+\ell_1} \text{ and }\Lambda=\frac{p\sigma}{(\sigma+\ell_1)p+\ell_2-\ell_1}.
\]
The random variable $\mw=g_N\cdot W_N$ is a non-negative martingale 
and converges almost surely to a limit $\mW{\infty}$. 
Furthermore, the random variable $W_N/N^{\Lambda}$ satisfies 
\[
W_N/N^{\Lambda}\to \frac{1}{\kappa}\mW{\infty}. 
\]
\end{lem}
\begin{remark}[Degenerated periodic triangular urns]
Setting $\ell_1=0$ and $\ell_2=\ell$ periodic triangular degenerate to P\'olya-Young urns. 
Thus, the results of Lemma~\ref{lem:Triangular} and the theorem stated below generalize our earlier results of Theorems~\ref{the1} and~\ref{the2}. Indeed, by $\ell_1=0$ and $\ell_2=\ell$ we readily re-obtain the previous values of $\psi$ and $\Lambda$: 
\[
\psi=p+\frac{\ell_2-\ell_1}{\sigma+\ell_1}=p+\frac{\ell}\sigma \text{ and }\Lambda=\frac{p\sigma}{(\sigma+\ell_1)p+\ell_2-\ell_1}
=\frac{p}{p+\ell/\sigma}.
\]
\end{remark}

\begin{theorem}[Periodic triangular urns: limit law and characterization]\label{theTriangular}
The raw moments of $W_N/N^{\Lambda}$ converge for $N\to\infty$. For $s\ge 1$ we have
\[
\E(W_N^s/\sigma^s N^{s\Lambda})\to \mu_s=p^{s\Lambda}\frac{\Gamma(s+\frac{w_0}\sigma)}{\Gamma(\frac{w_0}\sigma)}\prod_{r=0}^{p-1}\frac{\Gamma(\frac{r}\psi  +\frac{w_0+b_0}{\sigma_1\psi})}{\Gamma(\frac{r}\psi +\frac{w_0+b_0}{\sigma_1 \psi}+\frac{s\sigma}{\sigma_1\psi})}.
\]
The limit law $\frac{1}{\sigma\kappa}\mW{\infty}$ is uniquely determined by its moment sequence $(\mu_s)$. It has a density $f(x)$, $x>0$, given by
\[
f(x)=
\frac{\prod_{r=0}^{p-1}\Gamma(\frac{r}\psi +\frac{w_0+b_0}{\psi\sigma_1})}{\Gamma(\frac{w_0}\sigma)}
\cdot\sum_{j\ge 0}\frac{(-1)^j}{j!\prod_{\ell=0}^{p-1}\Gamma(\frac{\ell}\psi +\frac{b_0}{\sigma_1 \psi}-\frac{j\sigma}{\sigma_1\psi})}\ x^{j+w_0/\sigma-1}.
\]

For $\frac{\ell_2-\ell_1}{\sigma+\ell_1}\in\N_0$ the limit law can be decomposed into a product of a Mittag-Leffler law and generalized Gamma distributions, all being mutually independent:
\[
\frac{1}{\sigma\kappa}\mW{\infty}\law\psi^{\psi}\cdot\ML(\alpha,\beta,\gamma)\cdot
\prod_{r=p}^{\psi-1}\GG(a_r,b_r),
\]
with $\alpha=\frac{\sigma}{\sigma_1}$, $\beta=\frac{w_0}{\sigma_1}$ and $\gamma=\frac{b_0}{\sigma_1}$, 
as well as $a_r=\frac{r\sigma_1+w_0+b_0}{\sigma}$, $b_r=\frac{\sigma_1\psi}{\sigma}$.
In the general case $\frac{\ell_2-\ell_1}{\sigma+\ell_1}\notin\N_0$ the random variable can be presented in terms of the scaled local time $L$ of a noise-reinforced Bessel process with parameters $\alpha=p/\psi$ and $\beta =1/\psi$: let $T=\tilt_{1}(L)$, then
\[
\frac{1}{\sigma\kappa}\mW{\infty}=  \ML(\alpha,\beta,\gamma)\cdot \tilt_{\frac{w_0+b_0}{\sigma}-\frac{\sigma_1}{\sigma}+1}(T^{\frac{\sigma}{\sigma_1}}).
\]
\end{theorem}

\begin{remark}
We note that second order limit theorems for $W_N$ can also be obtained by considering 
the martingale tail sum for $\mw=g_N\cdot W_N$, similar to the P\'olya-Young urn.
\end{remark}

\begin{proof}[Proof of Lemma~\ref{lem:Triangular}]
Our starting point is Theorem~\ref{the1}, whose results are valid for the periodic triangular urns,
except for the asymptotics (see also Remark~\ref{rem:universal1}). 
Thus, we need the asymptotics $g_N$, with $T_N$ given by~\eqref{eq:TotalPerTria}.
We observe that
\[
T_N/\sigma_1=N +n\frac{\ell_2-\ell_1}{\sigma_1}+\frac{w_0+b_0}{\sigma_1},\quad N=np+k, \quad 0\le k\le p-1.
\]
Using the notation $\psi=p+\frac{\ell_2-\ell_1}{\sigma_1}$, we obtain
\begin{equation}
\begin{split}
\label{eq:AsympGTriangular1}
g_N&
=\bigg(\prod_{j=0}^{n-1}\prod_{r=0}^{p-1}\frac{j p+r + j\frac{\ell_2-\ell_1}{\sigma_1})+\frac{w_0+b_0}{\sigma_1}}{j p+r + j\frac{\ell_2-\ell_1}{\sigma_1})+\frac{w_0+b_0}{\sigma_1}+\frac{\sigma}{\sigma_1}}\bigg)
\cdot \prod_{r=0}^{k-1}\frac{np+r+n\frac{\ell_2-\ell_1}{\sigma_1}+\frac{w_0+b_0}{\sigma_1}}{np+r+n\frac{\ell_2-\ell_1}{\sigma_1}+\frac{w_0+b_0}{\sigma_1}+\frac{\sigma}{\sigma_1}}\\
&=\bigg(\prod_{j=0}^{n-1}\prod_{r=0}^{p-1}\frac{j +\frac{r}{\phi} + \frac{w_0+b_0}{\sigma_1\phi}}{j +\frac{r}{\phi} +\frac{w_0+b_0}{\sigma_1\phi}+\frac{\sigma}{\sigma_1\phi}}\bigg)
\cdot \prod_{r=0}^{k-1}\frac{n+\frac{r}{\phi} +\frac{w_0+b_0}{\sigma_1\phi}}{n+\frac{r}{\phi}+\frac{w_0+b_0}{\sigma_1\phi}+\frac{\sigma}{\sigma_1\phi}}.
\end{split}
\end{equation}

Writing the products in terms of the Gamma function gives
\begin{equation}
\begin{split}
\label{eq:AsympGTriangular2}
g_N&
=\bigg(\prod_{r=0}^{p-1}\frac{\Gamma(n+\frac{r}\psi +\frac{w_0+b_0}{\sigma_1 \phi})\Gamma(\frac{r}\phi  +\frac{w_0+b_0}{\sigma_1\phi}+\frac{\sigma}{\sigma_1\phi})}{\Gamma(\frac{r}\phi +\frac{w_0+b_0}{\sigma_1 \phi})\Gamma(n+\frac{r}\phi  +\frac{w_0+b_0}{\sigma_1\phi}+\frac{\sigma}{\sigma_1\phi}})\bigg)
\cdot \prod_{r=0}^{k-1}\frac{n+\frac{r}\phi +\frac{w_0+b_0}{\sigma_1 \phi}}{n+\frac{r}\phi  +\frac{w_0+b_0}{\sigma_1\phi}+\frac{\sigma}{\sigma_1\phi}}.
\end{split}
\end{equation}
Using Stirling's formula for the Gamma function~\eqref{eq:stirling}, 
we obtain the stated asymptotics, noting that
\[
p\cdot\frac{\sigma}{\sigma_1\phi}=\frac{p\sigma}{(\sigma+\ell_1)p+\ell_2-\ell_1},\quad
n^{-\Lambda}\sim p^{\Lambda}N^{-\Lambda}.
\]
Finally, we note that the asymptotics of $g_N$ imply
the almost sure convergence of $W_N/N^{\Lambda}$.
\end{proof}

\begin{proof}[Proof of Theorem~\ref{theTriangular}]
By Lemma~\ref{lem:Triangular}, see also Theorem~\ref{the1}, the rising factorial moments are given by
\[
\E(\auffak{(W_N/\sigma)}{s})=\auffak{(w_0/\sigma)}{s}\prod_{j=0}^{N-1}\frac{T_{j}+s\sigma}{T_{j}}.
\]
The asymptotics of $\prod_{j=0}^{N-1}\frac{T_{j}+s\sigma}{T_{j}}$ are readily obtained similar to the asymptotics of $g_N$~\eqref{eq:Def_gNTriangular} using Stirling's formula for the Gamma function~\eqref{eq:stirling}:

\begin{equation}
\begin{split}
\label{eq:AsympGTriangular3}
\prod_{j=0}^{N-1}\frac{T_{j}+s\sigma}{T_{j}}
&=\bigg(\prod_{r=0}^{p-1}\frac{\Gamma(n+\frac{r}\psi +\frac{w_0+b_0}{\sigma_1 \psi}+\frac{s\sigma}{\sigma_1\psi})\Gamma(\frac{r}\psi  +\frac{w_0+b_0}{\sigma_1\psi})}{\Gamma(\frac{r}\psi +\frac{w_0+b_0}{\sigma_1 \psi}+\frac{s\sigma}{\sigma_1\psi})\Gamma(n+\frac{r}\psi  +\frac{w_0+b_0}{\sigma_1\psi})}\bigg)\\
&\quad \times
 \prod_{r=0}^{k-1}\frac{n+\frac{r}\psi +\frac{w_0+b_0}{\sigma_1 \psi}+\frac{s\sigma}{\sigma_1\psi}}{n+\frac{r}\psi  +\frac{w_0+b_0}{\sigma_1\psi}}\sim n^{s\lambda}\mu_s = \frac{N^{s\lambda}}{p^{s\lambda}}\cdot\mu_s.
\end{split}
\end{equation}
The density of the limit law is obtained using an identical argument as in the proof of Theorem~\ref{the2} and therefore omitted. 
Similarly, Carleman's criterion is readily checked. Next we turn to the identification of the limit law. 
For $\frac{\ell_2-\ell_1}{\sigma+\ell_1}\in\N_0$ such that $\psi\in\N$ we extend the range of the product to $r=\psi-1$. 
This implies that
\begin{align*}
\mu_s&=  \sigma^s \frac{\Gamma(s+\frac{w_0}\sigma)}{\Gamma(\frac{w_0}\sigma)}
\cdot \frac{\Gamma(\frac{w_0+b_0}{\sigma_1})}{\Gamma(\frac{w_0+b_0+s \sigma}{\sigma_1})}
\psi^{s\psi} \prod_{r=p}^{\psi-1}\frac{\Gamma(\frac{r}\psi +\frac{w_0+b_0}{\psi\sigma_1}+\frac{s \sigma}{\sigma_1\psi})}{\Gamma(\frac{r}\psi +\frac{w_0+b_0}{\psi\sigma_1})}.
\end{align*}
The factors in the product are readily identified as generalized Gamma distributions with 
parameters $a_r=\frac{r\sigma_1+w_0+b_0}{\sigma}$, $b_r=\frac{\sigma_1\psi}{\sigma}$.
For the two prefactors
\[
\frac{\Gamma(s+\frac{w_0}\sigma)}{\Gamma(\frac{w_0}\sigma)}
\cdot \frac{\Gamma(\frac{w_0+b_0}{\sigma_1})}{\Gamma(\frac{w_0+b_0+s \sigma}{\sigma_1})}
\]
we use~\eqref{eq:momBML} and identify the limit law as a three parameter Mittag-Leffler distribution
with $\alpha=\frac{\sigma}{\sigma_1}$, $\beta=\frac{w_0}{\sigma_1}$ and $\gamma=\frac{b_0}{\sigma_1}$.

\smallskip

In the general case we use gain the random variable $T$ of Lemma~\ref{lem:PanKu2}, but taken to the power $\sigma/\sigma_1$:
\[
\E(T^{\frac{\sigma s}{\sigma_1}})=\Gamma(\frac{\sigma s}{\sigma_1}+1)\cdot\prod_{j=1}^{m}\frac{\Gamma(\frac{j\alpha}{m})}{\Gamma(\frac{s \sigma/\sigma_1+j)\alpha}{m})},\quad s\ge 1.
\]
We set $\alpha=p/\psi=p/(p+\frac{\ell_2-\ell_1}{\sigma+\ell_1})$, where $0<\alpha<1$, as $\ell_2\neq \ell_1$.
Next we use $m=p$, such that $\beta=\alpha/p=\frac1\psi$, and obtain 
\[
\E(T^{\frac{\sigma s}{\sigma_1}})=\Gamma(\frac{\sigma s}{\sigma_1}+1)\cdot\prod_{j=1}^{p}\frac{\Gamma(\frac{j}{\psi})}{\Gamma(\frac{(\sigma s/\sigma_1+j)}{\psi})},\quad s\ge 1.
\]
We tilt the moment sequence by $\frac{w_0+b_0}{\sigma}-\frac{\sigma_1}{\sigma}$ to obtain the sequence
\[
\frac{\Gamma(\frac{\sigma s}{\sigma_1}+\frac{w_0+b_0}{\sigma_1})}{\Gamma(\frac{w_0+b_0}{\sigma_1})}\cdot\prod_{j=1}^{p}\frac{\Gamma(\frac{j-1}{\psi} 
+\frac{w_0+b_0}{\sigma_1\psi})}{\Gamma(\frac{(s\sigma/\sigma_1+j-1)}{\psi} 
+\frac{w_0+b_0}{\sigma_1\psi})},\quad s\ge 1.
\]
Finally, in order to obtain the sequence $(\mu_s)$, we observe that 
a three parameter Mittag-Leffler distribution $X$
with $\alpha=\frac{\sigma}{\sigma_1}$, $\beta=\frac{w_0}{\sigma_1}$ and $\gamma=\frac{b_0}{\sigma_1}$, such that
\[
\frac{\Gamma(\frac{\sigma s}{\sigma_1}+\frac{w_0+b_0}{\sigma_1})}{\Gamma(\frac{w_0+b_0}{\sigma_1})}\cdot
\E(X^s) = \frac{\Gamma(s+\frac{w_0}\sigma)}{\Gamma(\frac{w_0}\sigma)}.
\]
\end{proof}

\section{Periodic Stirling permutations}
\label{sec:Stir}
Stirling permutations were introduced by Gessel and Stanley~\cite{GessStan1978}. A Stirling permutation is a permutation of
the multiset $\{1, 1, 2, 2, \dots , n, n\}$ such that, for each $i$, $1\le i \le n$, the elements occurring between the two
occurrences of $i$ are larger than $i$. E.g., $1122$, $1221$ and $2211$ are Stirling permutations, whereas the permutations $1212$ and $2112$ of $\{1, 1, 2, 2\}$ aren't. The name of these combinatorial objects is due to relations with the Stirling numbers, see~\cite{GessStan1978} for details. A natural generalization of Stirling permutations is to consider permutations of
a more general multiset $\{1^{d}, 2^{d}, \dots, n^{d}\}$, with $d\ge 2$, which are called $d$-Stirling permutations.
Note that $d=2$ yields exactly Stirling permutations, whereas setting $d=1$ gives just ordinary permutations. This class of $d$-Stirling permutations has been introduced already by Park~\cite{Park1994a,Park1994c,Park1994b} under the name $d$-multipermutations.
We note that by the Janson-Koganov bijection~\cite{Koganov1996,Jan2008} Stirling permutation are relation to plane-oriented recursive trees. 
The bijection of Gessel-Janson-Kuba-Panholzer~\cite{Gess2020,JKP2011} relates $d$-Stirling permutations with $(d+1)$-ary increasing trees; see also~\cite{Gess2020,JKP2011,KuPa2011Stir} for additional bijections, references to the literature and further information.

\smallskip

Let $p$ and $t$ denote positive integers, where $p\ge 2$. In the following we introduce periodic $d$-Stirling permutations with period $p$ and thickness $t$. There, after insertion steps $p$, $2p$, \dots additional $t$ boldfaced labels appear, 
so labels $p$, $2p$, etc. are thicker than there ordinary counterparts. 
As we will see in the following, these new boldfaced labels will relax the condition that the elements occurring between the two
occurrences of $i$ are larger than $i$. First, we collect a few basics. The order of a periodic $d$-Stirling permutation $\sigma$ is the number of different labels appearing in permutation. 
Let $N=np+k$, $0\le k<p$, $n\ge 0$. A periodic $d$-Stirling permutations of with thickness $t$ order $N$ 
has $N-n$ labels $\{1,\dots,N\}\setminus\{p,2p,\dots,np\}$ appearing $d$ times and $n$ labels $p,2p,\dots,np$ appearing $d+t$ times. 
Thus, the number of insertion places for label $N+1$ is given by
\[
q_N=N\cdot d +1 + n\cdot t
\]
Consequently, there are 
\[
Q_N=\prod_{j=1}^{N-1}q_j
\]
different periodic $d$-Stirling permutations with $N$ labels, period $p$, and thickness $t$.

\smallskip

Next we define the process generating all periodic $d$-Stirling permutation of order $N$ in a step-by-step fashion. 
\begin{defi}
Periodic $d$-Stirling permutation are created by the following procedure.
\begin{itemize}
	\item Step 1: we start with the string $1^d=11\dots 1$. 
	\item Step $N+1$, with $N\ge 1$: Given a permutation $\sigma$ of order $N$, we insert the complete string $(N+1)^d$ at any of the $q_N$ possible insertion places, obtaining a permutation $\tau$ of order $N+1$: insertion	in between two neighboring labels $\sigma_{m},\sigma_{m+1}$, $1\le m\le q_N-1$, creates
	\[
	\tau=\sigma_{1}\dots\sigma_{m}(N+1)^{d}\sigma_{m+1}\dots \sigma_{q_N},
	\]
	or on the left, $\tau=(N+1)^{d}\sigma$, or on the right, $\tau=\sigma (N+1)^{d}$.
	\item Step $N+1$: if $(N+1) \mod p=0$, then additionally put a (boldfaced/thick) String $(N+1)^{t}$ on the right of $\sigma$. 
\end{itemize}
\end{defi}
We note that for $t=0$ we re-obtain the usual $d$-Stirling permutations. For even $t$, we can add half of the labels on the left and half on the right for a more symmetric permutation. 
\begin{remark}[Creating random periodic $d$-Stirling permutation]
If we enter each new label $N+1$, with $N\ge 1$, uniformly at random according to the $q_N$ available positions, this gives a growth process generating random periodic $d$-Stirling permutation.
\end{remark}
\begin{remark}
We note that for each ordinary label $i$, $1\le i \le N$, the elements occurring between the two
occurrences of $i$ are still larger than $i$. However, if at least one label $i$ is boldfaced, and thus created at times $p$, $2p$, etc., 
then this rule is not valid. 
\end{remark}

\begin{example}
We consider a periodic Stirling permutation with $p=2$ and thickness $t=3$. 
The first few permutations are as follows: size one $11$, size two $2211{\boldsymbol{222}}$, $1221{\boldsymbol{222}}$, $1122{\boldsymbol{222}}$; here, the extra labels are printed in boldface. For size three permutations we have now eight different positions for insertion of $3^2$, compared to only five different positions for ordinary Stirling permutations. 
\end{example}

Next, we relate periodic $d$-Stirling permutations with period $p$ and thickness $t$ to increasing trees with immigration. 
We use $(d+1)$-ary increasing trees with the same parameter $p$, and use for the connectivity of the additional roots $\ell=t$. These
new roots are realized as special trimmed $t+1$-ary roots, where the first insertion place is inactive, compare with Figure~\ref{fig:Stir}.

\begin{theorem}
\label{bij:StirInc}
Periodic $k$-Stirling permutations with period $p\ge 2$ and thickness $t\ge 1$ are in bijection with increasing trees with immigration 
of period $p$ and parameter $\ell=t$. 
\end{theorem}

\begin{figure}[!htb]
\includegraphics[scale=0.7]{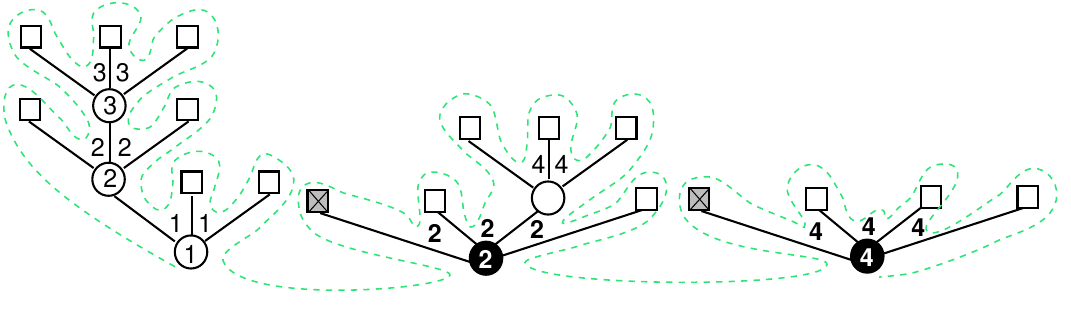}%
\caption{A periodic Stirling permutation $\sigma=233211\boldsymbol{22}44\boldsymbol{2}\boldsymbol{444}$ of order $N=4$, with period $p=2$ and thickness $t=3$ and the corresponding ternary increasing tree with immigration, obtaining by using the depth-first walk (a contour)}%
\label{fig:Stir}
\end{figure}

\begin{proof}
We use a depth-first walk of a rooted plane tree, starting at the root with the smallest label. We goes first to the leftmost child of the
root, explores that branch (recursively, using the same rules),
returns to the root, and continues with the next child of the
root, until there are no more children left. Then, we proceed with the remaining
We use a version of $(d+1)$-ary increasing trees with exterior nodes. 
Hence, at any time, any (interior) node has $d+1$
children, some of which may be exterior nodes. Between these $d+1$ edges
going out from a node labelled $v$, we place $d$ integers $v$. 
(Exterior nodes have no children and no labels.)
Now we perform the depth-first walk and code the $(d+1)$-ary increasing
tree by the sequence of the labels visited as we go around the tree 
(one may think of actually going around the tree like drawing
the contour). 
In other words, we add label $v$ to the code the $d$ first times we
return to node $v$, but not the first time we arrive there or the last
time we return. 
\end{proof}
 
We are interested in the distribution of the number 
of \emph{blocks} in a random periodic $d$-Stirling permutation $\sigma$ of order $N$. 
A block is a substring $a_{r}\dotsm a_{s}$ with $a_r=a_s$ of $\sigma$ that is maximal, i.e.~not contained in any larger
such substring. The size of a block is the number of labels creating the block, in our example $s+1-r$. There is obviously at most one block for every ordinary label $j=1,\dots,N$, extending from the first occurrence of $j$ to the last, where we distinguish between ordinary labels $1,2,\dots$, versus their thick counterparts, appearing periodically. By acclamation, the thick labels create blocks of minimal length $t$ on their on. We say that an ordinary label $j$ forms a block when this substring really is a block, i.e.~when it
is not contained in a string $i\dotsm i$ for some $i<j$.
In particular, $j$ forms a block if for any $i$ with $1\le i \le j-1$, there do not exist
indices $m_0,\dots m_{s+1}$ , with $1\le m_0<\dots<m_{s+1}\le q_N$, such that 
$\sigma_{m_0}=\sigma_{m_{s+1}}=i$ and
$\sigma_{m_1}=\dots=\sigma_{m_s}=j$.
It is easily seen by induction that any $\sigma$ has a unique
decomposition as a sequence of its blocks.
Note that if we add a string $(N+1)^{d}$ to a periodic $d$-Stirling permutation, this string
will either be swallowed by one of the existing blocks, or form a block on its own; the latter happens when it is added first,
last, or in a gap between two blocks.

\begin{example}
The periodic Stirling permutation $\sigma=233211\boldsymbol{22}44\boldsymbol{2}\boldsymbol{444}$ of order $N=4$
has block decomposition $\sigma=[2332][11][\boldsymbol{22}44\boldsymbol{2}][\boldsymbol{444}]$, 
one block of size four, one block of size two, one block of size four
and a block of size five.
\end{example}

By the bijection in Theorem~\ref{bij:StirInc} the size of the block formed by label one correspond to 
a restricted number of descendants of the original root, where the two outmost branches are neglected. 
Thus, Theorem~\ref{the1} and also Theorem~\ref{the2} applies and we obtain also obtain second order asymptotics for the size of the block $X_N$ formed by label one in a random periodic $d$-Stirling permutation. Actually, for $d>2$ one can do more, as the block generated by one has the form
\[
1\, X_{N,1} \,1\,X_{N,2}\,1\, \dots 1\, X_{N,d-1}\, 1
\]
has at $d-1\ge 2$ different, possibly empty, subblocks. They satisfy the identity
\[
X_N=d+\sum_{j=1}^{d-1}X_{N,j}.
\]
The distribution of these subblocks $X_{N,j}$ is exchangeable, and the (limiting) distribution can be described by our multivariate P\'olya-Young urn.

\smallskip 

The total number $S_N$ of blocks in a random periodic $d$-Stirling permutation is described by another urn
model. We use a periodic triangular urn
with period $p$ and replacement matrices
\[
M_1=M_2=\dots=M_{p-1}=
\left(
\begin{matrix}
1&d-1\\
0&d\\
\end{matrix}
\right),
\quad\text{and }
M_{p}=
\left(
\begin{matrix}
1&d-1+t\\
0&d+t\\
\end{matrix}
\right).
\]
We start with $w_0=2$ white balls and $d-1$ black balls, as in the beginning there are two positions where a new block can be created in $1^d$, and $d-1$ positions for entering $2^d$ into the block formed by the ones. The total number of blocks $S_N$ is related to $W_N$ by
\[
S_N=W_N-1+n,
\]
as for $N=np+k$ we created $n$ blocks with the boldfaced labels. Consequently, Lemma~\ref{lem:Triangular}, as well as Theorem~\ref{theTriangular}, can be applied 
leading to the limit law for $S_N$.

\section{Chinese restaurant process with competition}
The Chinese restaurant process with parameters $a$ and $\theta$ is a discrete-time stochastic process.
Its value at discrete time $n\in\N$ is one of the $B_n$ partitions of the set $[n]=\{1, 2, 3,\dots , n\}$, see Pitman~\cite{Pitman,Pitman2006}. The two parameters $a$ and $\theta$ are real values, satisfying the constraints $0< a <1$ and $\theta >-a$. Here, the $B_n$ denotes the Bell numbers, counting the number of partitions of an $n$-element set $B_0=B_1=1$, $B_2=2$, $B_3=5$, etc.\footnote{See sequence \href{http://oeis.org/A000110}{A000110}} in OEIS. One fancifully imagines a Chinese restaurant with an infinite number of tables, and each round table has an infinite number of seats. In the beginning the first customer takes place at the first table. At each discrete time step a new customer arrives and either joins one of the existing tables, or he takes place at the next empty table in line.
Each table corresponds to a block of a random partition. In the beginning at time $n = 1$, the trivial partition $\{ \{1\} \}$ is obtained with probability 1. Given a partition $T$ of $[n]$ with $|T|=m$ parts $t_i$, $1 \le i \le m \le n$, of sizes $|t_i|$. At time $n + 1$  the element $n + 1$ is either added to one of the existing parts $t_i\in T$ with probability
\[
\P\{n+1<_c t_i\}=\frac{|t_i|-a}{n+\theta},\quad 1\le i\le m,
\]
or added to the partition $T$ as a new singleton block with probability
\[
\P\{n+1<_c t_{|T|+1}\}=\frac{\theta+m \cdot a}{n+\theta}.
\]
This model thus assigns a probability to any particular partition $T$ of $[n]$. 
Above, the notation $n+1 <_c t_i$ stands for costumer $n+1$ connects to the table $t_i$.
Of classical interest is the total number of tables, as well the number of parts of size $j$ in a partition of $[n]$ generated by the Chinese restaurant process~\cite{KuPa2014,Pitman,Pitman2006}.
As pointed out in~\cite{KuPa2014}, this process can be embedded into a variant
of the growth process of generalized plane-oriented recursive trees with two different connectivity parameters $\ell>0$ and $\alpha>0$. Here, 
the root node $r$ attracts node with a probability proportional to $\ell$ plus its outdegree $d(r)$, and the other nodes attract with $\alpha$ instead of $\ell$. This allows to study properties of the Chinese restaurant process using analytic combinatorial tools, 
see~\cite{BKW2022,KuPa2014} for results. 

\begin{figure}[!htb]
\begin{center}
\includegraphics[width=.98\textwidth]{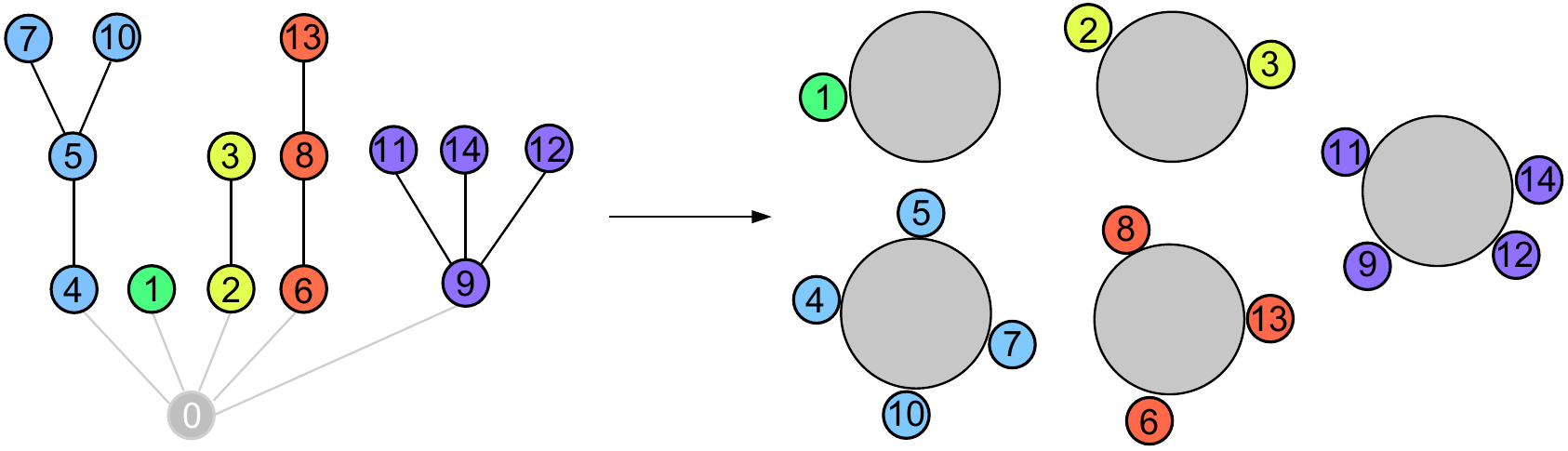} 
\end{center}
\caption{A plane-oriented recursive tree of size $15$, where the root node labeled zero has one size-one, one size-two, one size-three, and two size-four branches, as well as the corresponding table structure in the Chinese restaurant model.}
\label{fig:PortSubtreeSizes1}
\end{figure}


\smallskip

\subsection{Competition\label{subsec:comp}}
We introduce a generalization of the chinese restaurant process. In the chinese restaurant process with competition at each time step $N$ being a multiple of $p$ a new restaurant opens and attracts customers. Thus, from time one to time $p$ the process behaves just as the ordinary chinese restaurant process. Starting with time $p+1$, the probabilities change, as customers may sit in old restaurant, as well as the new restaurant(s). Especially at time $p+1$, a new customer may either go to the first restaurant, taking a table for himself or joining an existing group, or go to the just opened new restaurant, taking a table for himself. 

\begin{defi}[Chinese restaurant process with competition]
Let $p\in\N$ denote the duration, the time it takes to create a new restaurant and $N=np+k$, with $n\in\N$, $0\le k< p$. Furthermore, let
\[
c_N=np+k+(n+1)\theta=N+(n+1)\theta.
\]
\begin{itemize}
	\item At time $N=1$ the trivial partition $\{ \{1\} \}$ is obtained with probability 1.
	\item Time $N=np+k$: Given a partition $T$ of $[N]$ with $|T|=m$ parts $t_i$, $1 \le i \le m \le N$, of sizes $|t_i|$. At time $N + 1$  the element $N + 1$ is either added to one of the existing parts $t_i\in T$ with probability
\[
\P\{N+1<_c t_i\}=\frac{|t_i|-a}{c_N},\quad 1\le i\le m,
\]
or added to the partition $T$ as a new singleton block with probability
\[
\P\{N+1<_c t_{|T|+1}\}=\frac{(n+1)\theta+m \cdot a}{c_N}.
\]	
\end{itemize}
\end{defi}
\begin{remark}[Competition and singletons]
The creation of new restaurants and the competition is hidden in the value of $N$ and $n$. 
For example, at time $N=p+1$ the probability of a singleton jumps from being proportional to $\theta+m\cdot a$ 
to being proportional to $2\theta+m\cdot a$, as the new restaurant also attracts customers.
\end{remark}
\begin{remark}[No competition]
We can readily reobtain the ordinary restaurant process by setting $n=0$ and $p$ tending to infinity such that $np=0$, as no new restaurants are being created. Then $c_N=k+\theta$, $0\le k<\infty$, with $k$ taking the role of the discrete time, similar to $n$ in the ordinary setting.
\end{remark}
\begin{remark}[Cocktails]
It is also possible to add a so-called cocktail bar to the model (see~\cite{Moehle2021}), which may also attract customers. 
This corresponds then to increasing trees with immigration, starting with two root nodes. See Subsection~\ref{subsec:bar} for details. 
\end{remark}

\begin{figure}[!htb]
\begin{center}
\includegraphics[scale=0.55]{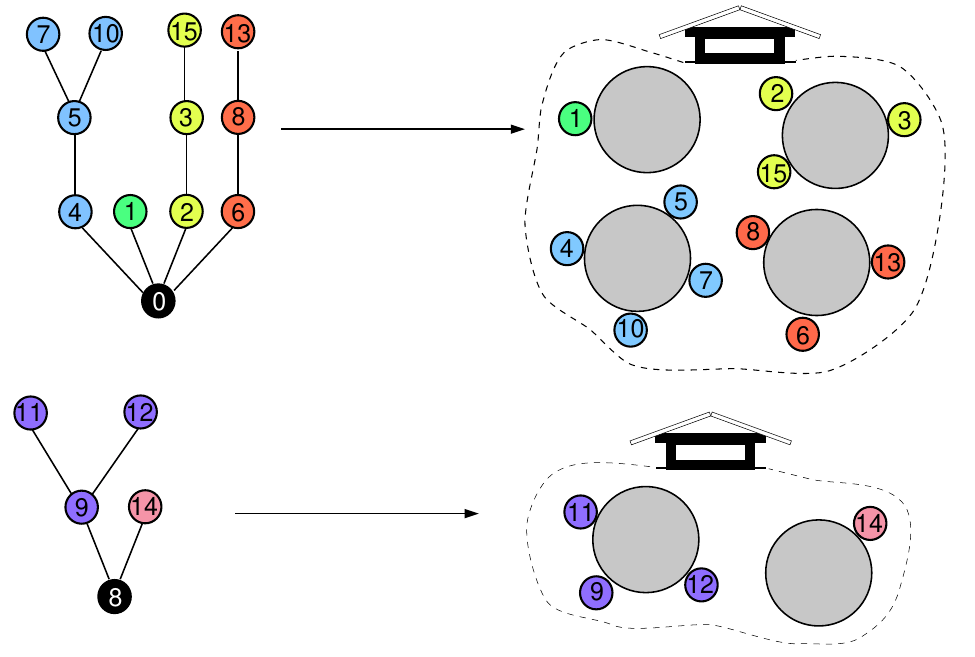} 
\end{center}
\caption{LHS: an increasing forest of size $16$, created by the (modified) growth process for increasing trees with immigration, $p=8$. RHS:the corresponding table structure with two 
competing chinese restaurants.}
\label{fig:PortSubtreeSizes2}
\end{figure}

There are plenty of parameters of interest in the chinese restaurant process with competition. One may study the number of tables occupied in the first restaurant, the second one, and so on, as well as the individual tables sizes in each restaurant.
Moreover, one may also consider the total number of tables in all restaurants, 
or the complete table structure in the original restaurant. These parameters are all intimately related to generalized P\'olya-Young urn models, as well as periodic triangular urns. In order to make this connection clear, we relate this new process to the increasing trees with immigration. 

\begin{theorem}[Chinese restaurant process with competition and increasing trees with immigration]
\label{ChineseThe}
Given a random partition of $\{1,\dots,N\}$ generated by the Chinese restaurant process with immigration, parameters $0<a<1$, $\theta>0$ 
and duration $p\in\N$. The random partition can be generated equivalently by the growth process of the family of a modified generalized
plane-oriented recursive trees $\mathcal{T}_{\alpha,\ell}$ with immigration parameter $p$, when generating such a tree of size $N+1$
with label set $\{0,1,\dots,N\}$. The parameters $a,\theta$ and $\alpha,\ell>0$, respectively, are related via
\begin{equation*}
	a=\frac{1}{1+\alpha},\qquad \theta=\frac{\ell}{1+\alpha}.
\end{equation*}
\end{theorem}

\begin{remark}[Creation of the first new root]
We note that the first new root is actually created after label $p$ is entered. 
As the original root is labeled zero, this requires $p$ insertions instead of $p-1$, is a modification of the growth process of increasing trees with immigration.
\end{remark}

\begin{proof}[Proof of Theorem~\ref{ChineseThe}]
Concerning the total connectivity~\eqref{eqn:Connectivity2}, we have to take care as there are $N+1$ labels, zero up to $N$, present.
Moreover, as the original root has connectivity $\ell+d(r)$, the total connectivity changes to,
\begin{equation}
\label{eqn:ConnectivityNew}
C_N
=N(1+\alpha) + (n+1)\ell.
\end{equation}
Given a size-$N$ forest $T$ of the family $\mathcal{T}_{\alpha,\ell}$ with labels $\{0,1,\dots ,N\}$. 
Assume that the original root has $m_0$ branches $t_i$, $1\le i\le m_0$, of sizes $|t_i|$, and the new roots labeled $jp$ created at times $p,2p,\dots, np$ each have $m_j$ branches $t_i$, $1\le i\le m_j$, $1\le j\le n$. Moreover, we denote with $m=\sum_{j=0}^{n}m_j$ the total number of branches.
At time $N + 1$ the element $N + 1$ is either attached to one of the 
existing non-root nodes $v$ with probability
\begin{equation*}
\begin{split}
\P\{N+1<_c v\}&=\frac{d(v)+\alpha}{C_N},
\end{split}
\end{equation*}
or to the $n+1$ different roots of the tree with probability 
\[
\P\{N+1<_c \text{root}_j\}=\frac{d(\text{root}_j)+\ell}{C_N}
=\frac{m_j+\ell}{C_N}.
\]
Consequently, element $N+1$ is a singleton with probability
\[
\sum_{j=0}^{n}\P\{N+1<_c \text{root}_j\}=\frac{m+(n+1)\ell}{N(1+\alpha) + (n+1)\ell}
=\frac{m\cdot\frac{1}{\alpha+1}+(n+1)\frac{\ell}{\alpha+1}}{N + (n+1)\frac{\ell}{\alpha+1}}.
\] 
Moreover, element $N+1$ is attached to one of the branches $t_i\in T$ with probability
\begin{equation*}
\begin{split}
\P\{N+1<_c t_i\}&=\sum_{v\in t_i}\P\{N+1<_c v\}=\sum_{v\in t_i}\frac{d(v)+\alpha}{C_N}\\
&=\frac{|t_i|-1+|t_i|\alpha}{N(1+\alpha) + (n+1)\ell}=\frac{|t_i|-\frac{1}{\alpha+1}}{N + (n+1)\frac{\ell}{\alpha+1}}.
\end{split}
\end{equation*}
Thus, setting $a=\frac{1}{1+\alpha}$ and $\theta=\frac{\ell}{1+\alpha}$ proves the stated relation.
\end{proof}

We note that quantities of interest include the number of tables of the first (original) restaurant, 
the table sizes, etc. By Theorem~\ref{ChineseThe} and Theorem~\ref{theOutdegree}, slightly modified according to anew initial value, the distribution of number of tables of the first (original) restaurant as well as the limit law is covered by Lemma~\ref{lem:Triangular} and Theorem~\ref{theTriangular}. The number of costumers of the first restaurant is random variable, covered by Theorems~\ref{the1} and~\ref{the2}. Concerning the table structure, one may also study the number of tables of a certain size, compare with the results of~\cite{KuPa2014}, where we anticipate a Mixed-Poisson type phase transition. However, we need a new periodic urn model with multiple colors for the branching structure, which by Theorem~\ref{ChineseThe} also determines the number of tables of a certain size. The urn model is reminiscent of the model used for the block sizes in $k$-Stirling permutations~\cite{KuPaStir2011}, but its analysis seems to be much more involved. 

\begin{defi}[Periodic urn model for branches]
\label{defUrnTables}
Given the period $p$. Consider a periodic urn models with balls of $j+2$ colours and let $(Z_{N,0},Z_{N,1}\dots,Z_{N,j+1})$ count
the number of balls of each color at time $N$. We use the ball addition matrices
\begin{equation*}
    M_k= \left(
    \begin{smallmatrix}
        1 & \alpha & 0 & \cdots & 0 & 0 & 0& 0 \\[-1ex]
        0 & -\alpha & 2(\alpha+1)-1 & \ddots & \ddots & \ddots & 0& 0 \\[-1ex]
        0 & 0 & -2(\alpha+1)+1 & 3(\alpha+1)-1 & \ddots & \ddots & 0 & 0\\[-1ex]
        \vdots & \ddots & \ddots & \ddots & \ddots &
        \ddots & \vdots & 0\\[-1ex]
        \vdots & \ddots & \ddots & \ddots & \ddots &
        \ddots & \vdots & 0\\[-1ex]
                0 & \ddots & \ddots & \ddots & 0 & -(j-1)(\alpha+1)+1 &  j(\alpha+1)-1& 0\\[-1ex]
        0 & \ddots & \ddots & \ddots & 0 & 0 & -j(\alpha+1)+1 & (j+1)(\alpha+1)-1\\
        0 & 0 & 0 & \cdots & 0 & 0 & 0& 1+\alpha
    \end{smallmatrix}
    \right)
\end{equation*}
for $1\le k\le p-1$, and
\begin{equation*}
M_{p}= \left(
    \begin{smallmatrix}
        1 & \alpha & 0 & \cdots & 0 & 0 & 0& \ell \\[-1ex]
        0 & -\alpha & 2(\alpha+1)-1 & \ddots & \ddots & \ddots & 0& \ell \\[-1ex]
        0 & 0 & -2(\alpha+1)+1 & 3(\alpha+1)-1 & \ddots & \ddots & 0 & \ell\\[-1ex]
        \vdots & \ddots & \ddots & \ddots & \ddots &
        \ddots & \vdots & \ell\\[-1ex]
        \vdots & \ddots & \ddots & \ddots & \ddots &
        \ddots & \vdots & \ell\\[-1ex]
                0 & \ddots & \ddots & \ddots & 0 & -(j-1)(\alpha+1)+1 & \ell j(\alpha+1)-1& \ell\\[-1ex]
        0 & \ddots & \ddots & \ddots & 0 & 0 & -j(\alpha+1)+1 & (j+1)(\alpha+1)-1+\ell\\
        0 & 0 & 0 & \cdots & 0 & 0 & 0& 1+\alpha+\ell
    \end{smallmatrix}
    \right)
\end{equation*}
The initial configuration of the urn is specified by
$(Z_{1,0},\dots,Z_{1,j+2})=(\ell,0,0,\dots,0)$.
\end{defi}
Next we state the connection between the branching structure (and thus the table structure) and the urn model.
\begin{prop}
The random variables $(Z_{N,1},\dots,Z_{N,j})$ of the urn model in the definition stated before 
are related to the number branches $B_{N,m}$ of size $m$ of the original root node in generalized plane-oriented recursive trees with immigration by
\[
B_{N,m}=\frac{Z_{N,m}}{m(\alpha+1)-1},\quad 1\le m\le j.
\]
\end{prop}
\begin{proof}
The color zero represents the root node. The random variable $Z_{N,0}$ encodes the connectivity of the original root, which increases by one
after a new node has been attached to it. The random variable $Z_{N,m}$, $1\le m\le j$, represent the connectivity of the subtrees of size $m$, given by $m(\alpha+1)-1$. Once a new node is attached to a size $m$ trees, the tree size changes to $m+1$, 
and thus we add a $(m+1)(\alpha+1)-1$ for type $m+1$, but remove the $m(\alpha+1)-1$ from color $m$. 
Finally, the last color $j+1$ is a dummy color (sometimes called super color or super ball in the literature), collecting
all other nodes, including the new roots. Consequently, at multiples of $p$ the dummy color increases by $\ell$, regardless of the sampled color. 
\end{proof}
The analysis of this model seems to be much more complicated, as we are also interested in letting $j$ tend to infinity. We plan to report on these distributional results in a future work.

\subsection{Adding a cocktail bar\label{subsec:bar}}
M\"ohle~\cite{Moehle2021} considered a chinese restaurant process, where an additional bar also attracts costumers. 
We generalize this process to restaurants with competition\footnote{Our choice of parameters slightly differs from~\cite{Moehle2021}.}.
\begin{defi}[Chinese restaurant process with and a single bar and competition]
Let $p\in\N$ denote the duration, the time it takes to create a new restaurant and $N=np+k$, with $n\in\N$, $0\le k< p$. 
We start with a single restaurant with tables and also a cocktail bar. 
All tables and the bar are assumed to have infinite capacity. At the beginning the first original restaurant and the bar is empty. 
Let the quantity $c_N$ be given by
\[
c_N=c_N=np+k+(n+1)\theta_1+\theta_2=N+(n+1)\theta_1+\theta_2,
\]
with $\theta_1$, $\theta_2$ denote two positive real parameters.
If at stage $N=np+k$ there are $b\ge 0$ customers at the bar and $m$ occupied tables
with $t_i>0$ customers at the $i$th table, $1\le i\le m$, the new customer $N+1$
sits
\begin{itemize}
	\item at the bar with probability $\displaystyle{\P\{N+1<_c \text{Bar}\}=\frac{b+\theta_2}{c_N}}$
	\item at table $t_i$, $1 \le i \le m \le N$ of sizes $|t_i|$ with probability
\[
\P\{N+1<_c t_i\}=\frac{|t_i|-a}{c_N},\quad 1\le i\le m,
\]
or at new new table with probability
\[
\P\{N+1<_c t_{m+1}\}=\frac{m \cdot a+(n+1)\theta_1}{c_N}.
\]	
\end{itemize} 
\end{defi}

\begin{remark}[Boundary cases]
In the boundary case $\theta_2=0$ the bar stays empty and can be neglected.
In the case $\theta_1=0$ all customers go to the bar, but no one stays at restaurant(s). 
\end{remark}

\begin{remark}[More cocktail bars]
One may readily generalize the process in various ways. For example, new bars could open at times $p$, $2p$,\dots,;
or bars open according to a period $p_2\neq p_1=p$. 
\end{remark}

We can readily define a new family $\mathcal{T}_{\alpha,\beta,\ell}$ of increasing trees with immigration, whose branching structure correspond to the new process with a bar. Here $\alpha,\beta$ and $\ell$ denote positive real parameters. Let the quantity $C_n$ be given by
\begin{equation*}
C_N
=(N+1)(1+\alpha) + n\ell-1 + \ell-\alpha+\beta,
=N(1+\alpha) + (n+1)\ell+\beta.
\end{equation*}
In the following we distinguish between three types of nodes: first, the root of the bar tree called $r_{\text{bar}}$ (labeled zero);
second, the original root labeled zero, as well as the root nodes created at times $p,2p,\dots$, and 
third the ordinary nodes labeled $1,2,\dots$ inserted into the forest. 
\begin{itemize}
\item Step $1$: The process starts with two roots labeled by $0$: a special "bar root" $r_{\text{bar}}$, as well as another root node. 
\item Step $i+1$ with $i\ge 0$: the node with label $i+1$ is attached to any previous non-root ordinary node $v$ with out-degree $d(v)$ 
in the increasing forest with $i$ ordinary non-root nodes with probabilities 
$\displaystyle{p(v)  =\frac{d(v)+\alpha}{C_i}}$. Additionally, the node with label $i+1$ is attached to any previously created new root nodes $v$, there are $\big\lfloor \frac{i}{p}\big\rfloor$ many of them, as well as the original root, with probability given by $\displaystyle{p(v)  = \frac{d(v)+\ell}{C_i}}$. The bar root node $v$ attracts node $i+1$ with probability
$\displaystyle{p(v)  =\frac{d(v)+\beta}{C_i}}$.
\item Step $i+1$, with $i+1\mod p=0$: a new root is created.
\end{itemize}

The correspondence between the increasing trees and the restaurant process with a bar follows now similar to the 
proof of Theorem~\ref{ChineseThe}. The only new ingredient is the connection between the tree rooted at the bar root node 
and the bar itself. The rest of the argument can be copied verbatim. Assume that the bar tree $B$ has consists of $b$ non-root nodes.
\begin{align*}
\P\{N+1<_c B\}&=\sum_{v\in B}\P\{N+1<_c v\}\\
&=\sum_{v\in B\setminus\{r_{\text{bar}}\}}\P\{N+1<_c v\}+\P\{N+1<_c r_{\text{bar}}\}\\
&=\sum_{v\in B\setminus\{r_{\text{bar}}\}}\frac{d(v)+\alpha}{C_N} + \frac{d(r_{\text{bar}})+\beta}{C_N}
= \frac{b+b\alpha+\beta}{N(1+\alpha) + (n+1)\ell+\beta}\\
&=\frac{b+\frac{\beta}{\alpha+1}}{N+(n+1)\frac{\ell}{\alpha+1}+\frac{\beta}{\alpha+1}}.
\end{align*}
Setting $a=\frac{1}{1+\alpha}$, $\theta_1=\frac{\ell}{1+\alpha}$ and $\theta_2=\frac{\beta}{1+\alpha}$ leads to the following theorem.

\begin{theorem}[Increasing trees and Chinese restaurant process with a single bar and competition]
Given a random partition of $\{1,\dots,N\}$ generated by the Chinese restaurant process with immigration and a bar, parameters $0<a<1$, $\theta_1,\theta_2>0$ and duration $p\in\N$. The random partition can be generated equivalently by the growth process of the family of a modified generalized plane-oriented recursive trees $\mathcal{T}_{\alpha,\beta,\ell}$, with immigration parameter $p$, when generating such a tree of size $N+1$ with label set $\{0,1,\dots,N\}$. This family of trees starts with two roots labeled by $0$: a special "bar root", as well as an ordinary root node. The parameters $a,\theta_1,\theta_2$ and $\alpha,\beta,\ell>0$, respectively, are related via
\begin{equation*}
	a=\frac{1}{1+\alpha},\quad \theta_1=\frac{\ell}{1+\alpha},\quad \theta_2=\frac{\beta}{1+\alpha}.
\end{equation*}
\end{theorem}

\section{Summary and Outlook}
\subsection{Summary}
We introduced and analyzed a generalized P\'olya-Young urn model. We obtained exact formulas for the distribution and the moments of the number of white balls, as well as limit laws. The limit laws involve the Beta distribution, the generalized Gamma distribution and also the local time of noise-reinforced Bessel process. 
We also added second order limit theorems, stating a central limit theorem for the martingale tails sums. Crucial for these results is the martingale structure of the urn model. Furthermore, we generalized ordinary increasing trees by allowing nodes to immigrate, leading to increasing forests. This new model was shown to be directly connected to generalized P\'olya-Young urn models. Moreover, we also looked at periodic triangular urns, which also generalize the P\'olya-Young urn model. For this class of urns the limit law also involves the Mittag-Leffler distribution with three parameters.

\begin{figure}[!htb]
\includegraphics[scale=0.65]{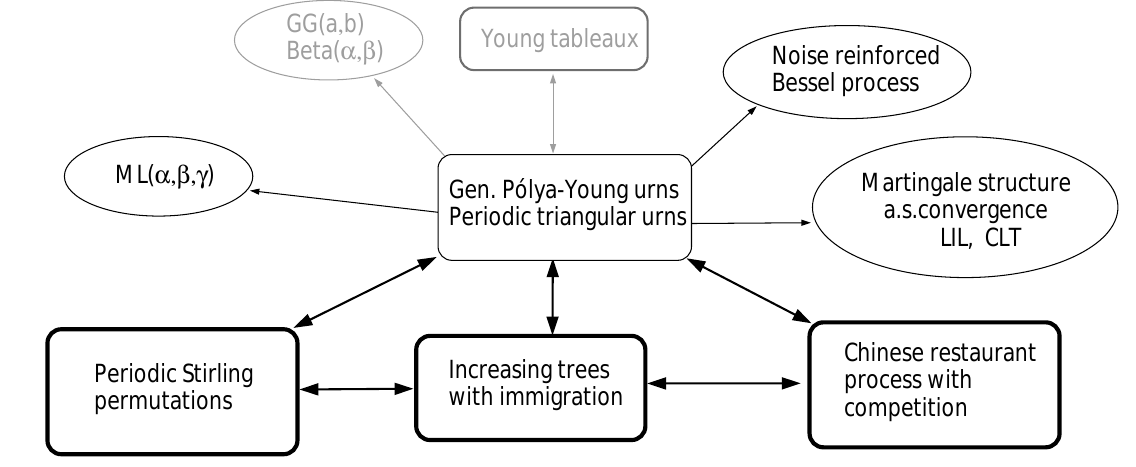}%
\caption{A schematic overview 
of the different connections for generalized P\'olya-Young urns and periodic triangular urn; we also collected (in gray) the connections established before in~\cite{BMW2020}.}%
\end{figure}

Finally, we related both urn models and increasing trees with immigration to a new version of the chinese restaurant process, where new restaurants open and compete with the previously established restaurants for customers. 
The analysis of many parameters in the chinese restaurant process, as well as increasing trees with immigration and also periodic $d$-Stirling permutations
can be reduced to the analysis of the corresponding urn model.

\begin{table}[!htb]
\renewcommand{\arraystretch}{2.5}
\begin{tabular}{|c|c|c|}
\hline
\textbf{Urn model} & \makecell{\textbf{Chinese restaurant process} \\ \textbf{with competition}} & \makecell{\textbf{Increasing trees}\\ \textbf{with competition}} \\
\hline
Generalized P\'olya-Young & \makecell{Number of costumers\\ of specific restaurant} & Descendants of label/root $j$ \\
\hline
Periodic Triangular & \makecell{Number of tables\\ of specific restaurant} & Degree of label/root $j$\\
\hline 
\makecell{Periodic urn \\for branches} & \makecell{Tables structure\\ of specific restaurant} &  \makecell{Branching structure of\\ a tree}\\
\hline
\end{tabular}
\caption{Connection between certain parameters in urn models, the chinese restaurant process, as well as increasing trees. 
Additional connections to periodic $d$-Stirling permutations are given in Section~\ref{sec:Stir}.}
\label{Table1}
\end{table}

\subsection{Outlook}
There are plenty of open questions and new directions for further investigations; see also the discussion in Subsection~\ref{subsec:comp}
and the connections summarized in Table~\ref{Table1}. From the viewpoint of urn models the effect of periodicity to urn models with a normal limit or so-called large urns (or mixture models) has not been investigated. It is also of interest to extend the current work to unbalanced (triangular) urn models.  Additionally, one may look at urn models with randomized entries and study the effect of the periodicity. Furthermore, the analysis of parameters in increasing trees with immigration, as well as the chinese restaurant process with competition is certainly of interest; in particular, a combinatorial model for increasing trees with immigration is highly desirable, as it would enable us to use analytic combinatorial frameworks and schemes.  As mentioned earlier, the author plans to report on limit laws for the branching structure and the table sizes in the chinese restaurant process in subsequent work, as well as for the different label-dependent limit laws 
in increasing trees with immigration.

\smallskip

Finally, we introduce an even more general urn model.
\begin{defi}[P\'olya urns associated to a sequence]
Given different replacement $m\times m$ matrices $A_1, A_2,\dots$ and a sequence of positive integers $(b_n)_{n\ge 1}$. 
A P\'olya urn associated to the sequence $(b_n)$ applies at time $n$ the replacement matrix $A_{b_n}$. 
\end{defi}

\begin{example}[Two-type ordinary periodic P\'olya urns]
\label{Ex:PoYo1}
Let $m=2$. We choose two replacement matrices $A_1$ and $A_2$ for two different colors. Let the integer $p>1$. We consider the periodic sequence
\[
b_n=\begin{cases}
2,\quad n\in p\N,\\
1, \quad n\notin p\N.
\end{cases}
\]
Thus, we apply $M_1=\dots=M_{p-1}=A_1$ and at every $p$th step we apply matrix $M_p=A_2$.
\end{example}

\begin{example}[Two-type P\'olya-Young urns associated to a sequence]
\label{Ex:PoYo2}
Let again $m=2$. We choose two replacement matrices $A_1$ and $A_2$, 
given by
\[
A_1=
\left(
\begin{matrix}
\sigma&0\\
0&\sigma\\
\end{matrix}
\right),
\quad\text{and }
A_2=
\left(
\begin{matrix}
\sigma&\ell\\
0&\sigma+\ell\\
\end{matrix}
\right).
\]
and a sequence $(b_n)$ with $b_n\in \{1,2\}$. We note that for all such sequences the finite time results of Theorem~\ref{the1} stay valid with $N$ replaced by $n$, except for the asymptotics of $g_N$. However, we expect that the asymptotic analysis of $g_N$ for general classes of sequences $b_n$ is much involved, as the total number of balls at time $n$ becomes more irregular. 
For example, consider the (shifted) Thue-Morse sequence $b_n=t_n+1$, with $t_0=0$, $t_{2n}=t_n$, $t_{2n+1}=1-t_n$. 
Moreover, an interesting question is at what frequency the effect of $A_2$ is negligible and to determine the corresponding sequences $(b_n)$and their properties. 
\end{example}
Similarly, we can readily define increasing trees with immigration according to sequence $(b_n)$, 
a corresponding model of a chinese restaurant process with competition, or $d$-Stirling permutations associated to $(b_n)$, 
and their parameters correspond to our urn models of Examples~\ref{Ex:PoYo1} and~\ref{Ex:PoYo2}.

\section*{Acknowledgments}
The author thanks Cyril Banderier and Michael Wallner for discussions about P\'olya-Young urns
and also about their connection to urn models with multiple drawings, which initiated this research. 
Furthermore, the author is indebted to Alois Panholzer for encouragement and interesting discussions about increasing trees with immigration, label-based parameters and urns models.

\section*{Declarations of interest}
The authors declare that they have no competing financial or personal interests that influenced the work reported in this paper. 

\bibliographystyle{siam}
\bibliography{PolyaYoungUrns-refs}{}


\end{document}